\documentclass[12pt,reqno]{amsart}
\newcommand\ignore[1]{}
\input amssym.def
\usepackage[raiselinks=false,colorlinks=true,citecolor=blue,urlcolor=blue,linkcolor=blue,bookmarksopen=true,pdftex]{hyperref}
\usepackage{booktabs,multirow,lscape,longtable} 
\usepackage{mathtools} 
\usepackage{tikz} 
\numberwithin{equation}{section}
\numberwithin{equation}{subsection} 
\newtheorem{theorem}{Theorem}[section]

\newtheorem{lemma}[theorem]{Lemma}
 
\newtheorem{remark}[theorem]{Remark}

\begin{document} 
\baselineskip=15.5pt

\title
[Irreducible unitary representations with non-zero $(\frak{g}, K)$-cohomology]{Irreducible unitary representations with non-zero relative Lie algebra cohomology of the Lie group $SO_0(2,m)$}  
\author{Ankita Pal, Pampa Paul }
\address{Department of Mathematics, Presidency University, 86/1 College Street, Kolkata 700073, India}
\email{ankita.maths1995@gmail.com \\ pampa.maths@presiuniv.ac.in }

\thispagestyle{empty}
\date{}

\begin{abstract}
By a theorem of D. Wigner, an irreducible unitary representation with non-zero $(\frak{g},K)$-cohomology has trivial infinitesimal character, and hence up to unitary equivalence, 
these are finite in number. We have determined the number of equivalence classes of these representations and the Poincar\'{e} polynomial of cohomologies of these 
representations for the Lie group $SO_0(2,m)$ for any positive integer $m.$ We have also determined, among these, which are discrete series representations and holomorphic 
discrete series representations. 

\end{abstract}
\maketitle

\noindent 
\section{Introduction} 

Let $G$ be a connected semisimple Lie group with finite centre, and $K$ be a maximal compact subgroup of $G$ with Cartan involution $\theta.$ 
The differential of $\theta$ at identity is denoted by the same notation $\theta.$
Let $\frak{g}_0$ be the Lie algebra of $G, \frak{k}_0$ be the subalgebra of $\frak{g}_0$ corresponding to the Lie subgroup $K$ of $G, 
\frak{h}_0$ be a $\theta$-stable fundamental Cartan subalgebra of $\frak{g}_0,$ and $\frak{g}=\frak{g}_0^\mathbb{C}, \frak{h}=\frak{h}_0^\mathbb{C}.$ 
Corresponding to a $\theta$-stable parabolic subalgebra $\frak{q}$ of $\frak{g}_0$ containing $\frak{h}_0,$ and a linear function $\lambda$ on $\frak{h}$ in a certain good range, there is a 
cohomologically induced module $A_\frak{q}(\lambda),$ which is an irreducible unitary representation of $G$ with infinitesimal character $\chi_\lambda.$ 
These representations include all discrete series representations if rank$(G)=$rank$(K).$ We are 
interested in those cohomologically induced module $A_\frak{q}(\lambda)$ for which the infinitesimal character $\chi_\lambda$ is trivial, that is $\chi_\lambda$ is the infinitesimal character of the 
trivial representation of $G$, and we denote it by $A_\frak{q}.$ 

 By a theorem of D. Wigner, an irreducible representation with non-zero $(\frak{g},K)$-cohomology has trivial infinitesimal character. 
Hence there are only finitely many irreducible unitary representations with non-zero $(\frak{g}, K)$-cohomology. 
In fact, the irreducible unitary representations with non-zero relative Lie algebra cohomology are exactly the irreducible unitary representations $A_\frak{q}.$ See \S \ref{general} for more 
details. So the representations $A_\frak{q}$ are important on its own. Apart from that Borel \cite{borelc} has conjectured that an irreducible unitary representation with 
non-zero $(\frak{g},K)$-cohomology is an automorphic representation for a suitable uniform discrete subgroup of $G.$ Millson and Raghunathan \cite{mira} have proved this conjecture 
for the group $G=SO(n,1),$ by constructing geometric cycles and using Matsushima's isomophism \cite{matsushima}. So the representations $A_\frak{q}$ are possible candidates of automorphic representations of $G.$ 

Collingwood \cite{collingwood} has determined the representations $A_\frak{q}$ and computed cohomologies of these representations of $Sp(n,1),$ and the real rank one real form of $F_4.$ 
Li and Schwermer \cite{lisch} have determined the representations $A_\frak{q}$ and cohomologies of these representations for the connected non-compact real Lie group of 
type $G_2.$ Mondal and Sankaran \cite{mondal-sankaran2} have determined certain representations $A_\frak{q}$ of Hodge type $(p,p)$ when $G/K$ is an irreducible Hermitian symmetric space. 
If $G$ is a complex simple Lie group, the number of equivalence classes of the representations $A_\frak{q},$ and Poincar\'{e} polynomials of cohomologies of some of these representations have been 
determined in \cite{paul}. In this article, we have determined the number of equivalence classes of the representations $A_\frak{q},$ and Poincar\'{e} polynomials of cohomologies of these representations,  
when $G=SO_0(2,m)$ for any positive integer $m.$ The main results are stated as follows: 

\begin{theorem}\label{th1}
(i) If $A$ is the number of equivalence classes of irreducible unitary representations with non-zero $(\frak{g}, K)$-cohomology of the Lie group $SO_0(2,m)(m \in \mathbb{N})$, then  
\[ A = 
\begin{cases}
l(l+2)  & \textrm{if } m=2l-1, \\
l^2+4l-3 & \textrm{if } m =2l-2. \\
\end{cases}
\] \\ 
(ii) An $A_\frak{q}$ is unitarily equivalent to a discrete series representation of $SO_0(2,m)$ with trivial infinitesimal character {\it if and only if} 
$(-\Delta(\frak{u} \cap \frak{p}_-)) \cup \Delta(\frak{u} \cap \frak{p}_+) = \Delta_n^+.$ Also if $m \neq 2,$
an $A_\frak{q}$ is unitarily equivalent to a holomorphic discrete series representation of $SO_0(2,m)$ with trivial infinitesimal character {\it if and only if} 
$\Delta(\frak{u} \cap \frak{p}_+)=\phi \textrm{ or } \Delta_n^+.$ If $D$(respectively, $D_h$) is the number of equivalence classes of discrete series representations 
(respectively, holomorphic discrete series representations) of $SO_0(2,m)$ with trivial infinitesimal character, then 
$D = 2l$ if $m=2l-1,$ or $2l-2.$ Also $D_h=2$ if $m \neq 2.$ If $m=2,$ then $D_h=4.$ 
\end{theorem}

We have also determined Poincar\'{e} polynomials of cohomologies of these representations in the Table \ref{b-table}, Table \ref{d-table}. The proof of Th.\ref{th1} is given in 
\S \ref{proof}. We have used Remark 3.3(iii) of \cite{mondal-sankaran2} to prove Th.\ref{th1}(i), and an alternative approach of $\theta$-stable parabolic subalgebras to 
determine the Poincar\'{e} polynomial of cohomologies of these representations.

\noindent
\section{Irreducible unitary representations with non-zero $(\frak{g}, K)$-cohomology}\label{general}

Let $G$ be a connected semisimple Lie group with finite centre, and $K$ be a maximal compact subgroup of $G$ with Cartan involution $\theta.$ The differential of $\theta$ at identity is denoted 
by the same notation $\theta.$ Let $\frak{g}_0=$Lie$(G),\frak{k}_0=$Lie$(K),$ and $\frak{g}_0=\frak{k}_0 \oplus \frak{p}_0$ be the Cartan decomposition corresponding to $\theta.$ Let 
$\frak{g}=\frak{g}_0^\mathbb{C},\frak{k}=\frak{k}_0^\mathbb{C} \subset \frak{g}, \frak{p}=\frak{p}_0^\mathbb{C} \subset \frak{g}.$ 

A $\theta$-stable parabolic subalgebra of $\frak{g}_0$ is a parabolic subalgebra $\frak{q}$ of $\frak{g}$ such that 
(a) $\theta(\frak{q}) = \frak{q}$, and (b) $\bar{\frak{q}} \cap \frak{q}= \frak{l}$ is a Levi subalgebra of $\frak{q}$; 
where $\bar{\ }$ denotes the conjugation of $\frak{g}$ with respect to $\frak{g}_0$. By (b), $\frak{l}$ is the 
complexification of a real subalgebra $\frak{l}_0$ of $\frak{g}_0$. Also $\theta(\frak{l}_0) = \frak{l}_0$ and and $\frak{l}_0$ contains a maximal 
abelian subalgebra $\frak{t}_0$ of $\frak{k}_0$. Then $\frak{h}_0 = \frak{z}_{\frak{g}_0} (\frak{t}_0)$ is a $\theta$-stable Cartan 
subalgebra of $\frak{g}_0.$ Let $\frak{t}=\frak{t}_0^\mathbb{C} \subset \frak{k},$ and $\frak{h}=\frak{h}_0^\mathbb{C} \subset \frak{g}.$ 
Note that $\frak{t},\frak{h}$ are Cartan subalgebras of $\frak{k},\frak{g}$ respectively and $\frak{h} \subset \frak{q}.$ Let $\frak{u}$ be the nilradical of $\frak{q}$ so that $\frak{q} = \frak{l} \oplus \frak{u}.$ Then 
$\frak{u}$ is $\theta$-stable and so $\frak{u} = (\frak{u} \cap \frak{k}) \oplus (\frak{u} \cap \frak{p}).$ 

 If $V$ is a finite dimensional complex $A$-module, where $A$ is an abelian Lie algebra; we denote by $\Delta (V)$( or by $\Delta (V , A)$), 
the set of all non-zero weights of $V;$ by $V^\alpha,$ the weight space of $V$ corresponding to a weight $\alpha \in \Delta(V);$ and 
by $\delta (V)$ (or by $\delta (V , A)$), $1/2$ of the sum of elements in $\Delta (V)$ counted 
with their respective multiplicities. Fix a maximal abelian subspace $\frak{t}_0$ of $\frak{k}_0$ and $\frak{t}=\frak{t}_0^\mathbb{C} \subset \frak{k}.$ 
Since $\frak{k},\frak{g}$ are  $\frak{t}$-modules (under the adjoint action), we have 
\[\frak{k}=\frak{t}\oplus \sum_{\alpha \in \Delta(\frak{k},\frak{t})} \frak{k}^\alpha, \textrm{and } 
 \frak{g}=\frak{h}\oplus \sum_{\alpha \in \Delta(\frak{g},\frak{t})} \frak{g}^\alpha.\]
Note that $\Delta(\frak{k},\frak{t})$ is actually the set of all non-zero roots of $\frak{k}$ relative to the Cartan subalgebra $\frak{t}.$ 
Choose a system of positive roots $\Delta_\frak{k}^+$ in $\Delta(\frak{k},\frak{t}).$
 If $x \in i\frak{t}_0$ be such that $\alpha(x) \ge 0$ for all $\alpha \in \Delta_\frak{k}^+,$ then 
$\frak{q}_x= \frak{h}\oplus \sum_{\alpha \in \Delta(\frak{g},\frak{t}),\alpha(x) \ge 0} \frak{g}^\alpha$ 
is a $\theta$-stable parabolic subalgebra of $\frak{g}_0$, $\frak{l}_x= \frak{h}\oplus \sum_{\alpha \in \Delta(\frak{g},\frak{t}),\alpha(x)= 0} \frak{g}^\alpha$ is the 
Levi subalgebra of $\frak{q}_x,$ and $\frak{u}_x= \sum_{\alpha \in \Delta(\frak{g},\frak{t}),\alpha(x) > 0} \frak{g}^\alpha$ is the nilradical of $\frak{q}_x.$ If $\frak{q}$ is a 
$\theta$-stable parabolic subalgebra of $\frak{g}_0$, there exists $k \in K$ such that $Ad(k)(\frak{q})=\frak{q}_x.$ 

 Now associated with a $\theta$-stable parabolic subalgebra $\frak{q}$, we have an irreducible unitary representation 
$\mathcal{R}^S _\frak{q} (\mathbb{C}) = A_\frak{q}$ of $G$ with trivial infinitesimal character, where $S = \textrm{dim}
(\frak{u} \cap \frak{k}).$ The associated $(\frak{g}, K)$-module $A_{\frak{q}, K}$ contains an 
irreducible $K$-submodule $V$ of highest weight (with respect to $\Delta ^+ _\frak{k}$) 
$2 \delta (\frak{u} \cap \frak{p}, \frak{t}) = 
\sum_{\beta \in \Delta (\frak{u} \cap \frak{p}, \frak{t})} \beta $ and 
it occurs with multiplicity one in $A_{\frak{q}, K}$. Any other irreducible $K$-module that occurs in $A_{\frak{q}, K}$ has highest weight 
of the form $2 \delta (\frak{u} \cap \frak{p}, \frak{t}) + 
\sum_{\gamma \in \Delta (\frak{u} \cap \frak{p}, \frak{t})} n_\gamma \gamma,$
with $n_\gamma$ a non-negative integer \cite[Th. 2.5]{voganz}.  
The $(\frak{g}, K)$-modules $A_{\frak{q} , K}$ were first constructed, in general, by Parthasarathy \cite{parthasarathy1}. 
Vogan and Zuckerman \cite{voganz} gave a construction of the $(\frak{g}, K)$-modules $A_{\frak{q} , K}$ via cohomological induction and 
Vogan \cite{vogan} proved that these are unitarizable. Define an equivalence relation on the set of all $\theta$-stable parabolic subalgebras of $\frak{g}_0,$ by $\frak{q}$ is equivalent 
to $\frak{q}'$ if either $Ad(k)(\frak{q})=\frak{q}',$ for some $k \in K,$ or $\frak{u} \cap \frak{p} = \frak{u}' \cap \frak{p}.$ Also unitary equivalence is an equivalence relation 
on the set of all irreducible unitary representations $A_\frak{q}.$ Then the set of all equivalence classes of $\theta$-stable parabolic subalgebras are in one-one correspondence with 
the set of all equivalence classes of the irreducible unitary representations $A_\frak{q}$ \cite[Prop. 4.5]{riba}. 

 If $\frak{q}$ is a $\theta$-stable parabolic subalgebra of $\frak{g}$, then the Levi subgroup $L = \{g \in G : \textrm{Ad}(g) (\frak{q}) = 
\frak{q} \}$ is a connected reductive Lie subgroup of $G$ with Lie algebra $\frak{l}_0.$ As $\theta(\frak{l}_0) = \frak{l}_0, L \cap K$ is a maximal 
compact subgroup of $L$. One has 
\[ H^r (\frak{g}, K; A_{\frak{q}, K}) \cong H^{r-R(\frak{q})} (\frak{l}, L\cap K ; \mathbb{C}), \]
where $R(\frak{q}) := \textrm{dim}(\frak{u} \cap \frak{p})$. Let $Y_\frak{q}$ denote the compact dual of the 
Riemannian globally symmetric space $L/{L\cap K}$. Then $H^r (\frak{l}, L\cap K ; \mathbb{C}) \cong 
H^r (Y_\frak{q} ; \mathbb{C})$. And hence 
\[  H^r (\frak{g}, K; A_{\frak{q}, K}) \cong H^{r-R(\frak{q})} (Y_\frak{q} ; \mathbb{C}).\]
If $P_\frak{q}(t)$ denotes the Poincar\'{e} polynomial of $ H^* (\frak{g}, K; A_{\frak{q}, K})$, then by the above result, we 
have 
\[ P_\frak{q}(t) = t^{R(\frak{q})} P(Y_\frak{q} , t). \]
 Conversely, if $\pi$ is an irreducible unitary represention of $G$ with non-zero $(\frak{g},K)$-cohomology,  then $\pi $ is unitarily equivalent 
to $A_\frak{q}$ for some $\theta$-stable parabolic subalgebra $\frak{q}$ of $\frak{g}_0$ \cite[Th. 4.1]{voganz}.  See also \cite{vogan97} for a beautiful description of 
the theory of $(\frak{g}, K)$-modules $A_{\frak{q} , K}.$ 

 If rank$(G) =$ rank$(K)$ and $\frak{q}$ is a $\theta$-stable Borel subalgebra that is, $\frak{q}$ is a Borel subalgebra of $\frak{g}$ 
containing a Cartan subalgebra of $\frak{k}$, then $A_\frak{q}$ is a 
discrete series representation of $G$ with trivial infinitesimal character. In this case, $R(\frak{q})= \frac{1}{2} \textrm{ dim}(G/K)$,  
$L$ is a maximal torus in $K$ and hence 
\[H^r (\frak{g}^\mathbb{C}, K; A_{\frak{q}, K}) =
\begin{cases}
0 & \textrm{if } r \neq R(\frak{q}), \\
\mathbb{C} & \textrm{if } r = R(\frak{q}). 
\end{cases}
\]
If we take $\frak{q} = \frak{g}$, then $L=G$ and $A_\frak{q} = \mathbb{C}$, the trivial representation of $G$. If $G/K$ is Hermitian symmetric, choose a Borel-de Siebenthal positive 
root system in $\Delta(\frak{g},\frak{t})$ containing $\Delta_\frak{k}^+,$ and a unique non-compact simple root $\nu;$ and define 
$\frak{p}_+=\sum_{\beta \in \Delta(\frak{g},\frak{t}), n_\nu(\beta)=1}\frak{g}^\beta, 
\frak{p}_-=\sum_{\beta \in \Delta(\frak{g},\frak{t}),n_\nu(\beta)=-1} \frak{g}^\beta;$ where $n_{\nu}(\beta)$ is the coefficient of $\nu$ in the decomposition of $\beta$ 
into simple roots.  Then $\frak{p}=\frak{p}_+\oplus\frak{p}_-, \frak{u}\cap\frak{p}=(\frak{u}\cap\frak{p}_+)\oplus(\frak{u}\cap\frak{p}_-).$ 
Define $R_+(\frak{q})=$dim$(\frak{u}\cap\frak{p}_+), R_-(\frak{q})=$dim$(\frak{u}\cap\frak{p}_-).$ So $R(\frak{q})=R_+(\frak{q})+R_-(\frak{q}).$ One has a Hodge decomposition 
\[H^r (\frak{g}, K; A_{\frak{q}, K}) = \oplus_{p+q=r} H^{p,q} (\frak{g}, K; A_{\frak{q}, K}) =  H^{p,q} (\frak{g}, K; A_{\frak{q}, K}) \cong 
H^{p-R_+(\frak{q}),q-R_-(\frak{q})} (Y_\frak{q} ; \mathbb{C}) ; \]
where $p+q=r, p-q=R_+(\frak{q})-R_-(\frak{q}).$ See \cite[Ch. II, \S 4]{borel-wallach}, \cite{gh}, \cite{voganz}. The pair $(R_+(\frak{q}), R_-(\frak{q}))$ 
is referred to be the {\it Hodge type} of the representation $A_\frak{q}.$ 

\noindent
\section{Discrete series representations} 

We follow the notations from the previous section. 
Assume that rank$(G)=$ rank$(K),$ so that $G$ admits discrete series representation. A non-singular linear function $\lambda$ on $i\frak{t}_0$ relative to $\Delta(\frak{g},\frak{t})$ dominant 
with respect to $\Delta_\frak{k}^+,$ defines uniquely a positive root system $\Delta_\lambda^+$ of $\Delta(\frak{g},\frak{t})$ containing $\Delta_\frak{k}^+.$ Define 
$\delta_\frak{g} =\frac{1}{2} \sum_{\alpha \in \Delta_\lambda^+}\alpha, \delta_\frak{k} = \frac{1}{2}\sum_{\alpha \in \Delta_\frak{k}^+}\alpha.$ If $\lambda +\delta_\frak{g}$ is 
analytically integral(that is, $\lambda +\delta_\frak{g}$ is the differential of a Lie group homomorphism on the Cartan subgroup of $G$ corresponding to $\frak{t}_0$), then there exists a 
discrete series representation $\pi_\lambda$ with infinitesimal character $\chi_\lambda$(it is the character of the Verma module of $\frak{g}$ with highest weight 
$\lambda-\delta_\frak{g}$); the associated $(\frak{g}, K)$-module $\pi_{\lambda,K}$ contains an 
irreducible $K$-submodule with highest weight $\Lambda=\lambda+\delta_\frak{g}-2\delta_\frak{k}$ and 
it occurs with multiplicity one in $\pi_{\lambda, K}.$ Any other irreducible $K$-module that occurs in $\pi_{\lambda, K}$ has highest weight 
of the form $\Lambda + \sum_{\alpha \in \Delta_\lambda^+} n_\alpha \alpha,$ with $n_\alpha$ a non-negative integer. Upto unitary equivalence these are all discrete series representations of $G.$ 
This $\lambda$ is called the {\it Harish-Chandra parameter}, $\Lambda$ is called the {\it Blattner parameter} of the discrete series representation $\pi_\lambda.$ 
The positive root system $\Delta_\lambda^+$ is called the {\it Harish-Chandra root order} corresponding to $\lambda.$ If $G/K$ is Hermitian symmetric, then $\pi_\lambda$ is a 
holomorphic discrete series representation {\it if and only if} the Harish-Chandra root order corresponding to $\lambda$ is a Borel-de Siebenthal positive root system. See \cite{hc1}, \cite{knapp}.

\noindent
\section{Irreducible unitary representations with non-zero $(\frak{g}, K)$-cohomology of the Lie group $SO_0(2,m)$} 
 
Let $I_{2,m} = \left(
\begin{array}{ccc}
-I_2 & 0 \\
0 & I_m \\
\end{array}
\right)$, where $I_m$ denotes the identity matrix of order $m$. 
Let $G$ be the group $SO_0(2, m)$, the connected component of the group $\{g \in SL(m+2, \mathbb{R}) : g^t I_{2,m} g = I_{2,m} \}$. Then $G$ is a Lie group  
with Lie algebra $\frak{g}_0 = \frak{so}(2,m)$ = $\{ \left(
\begin{array}{ccc}
X_1 & X_2 \\
X_2^t & X_3 \\
\end{array}
\right): \textrm{all }X_i \textrm{ real}, X_1, X_3 \textrm{ skew symmetric of order } 2 \textrm{ and } m \textrm{ respectively}, X_2 \textrm{ arbitrary}\}$. 
The map $\theta : G \longrightarrow G$ given by $\theta(g) = I_{2,m} g I_{2,m}$ for all $g \in G$ is a Cartan involution with maximal compact subgroup $K= \{ \left(
\begin{array}{ccc}
A & 0 \\
0 & B \\
\end{array}
\right): A \in SO(2), B \in SO(m)\} \cong SO(2) \times SO(m)$. The differential of $\theta$ at the identity element of $G$ is the map $ X \mapsto  I_{2,m} X I_{2,m}$ for all $X \in \frak{g}_0$, 
and is denoted by the same notation $\theta : \frak{g}_0 \longrightarrow \frak{g}_0$. Then $\theta : \frak{g}_0 \longrightarrow \frak{g}_0$ is a Cartan involution and 
$\frak{g}_0 = \frak{k}_0 \oplus \frak{p}_0$ is the Cartan decomposition corresponding to $+1$ and $-1$-eigenspaces of $\theta$. Note that $\frak{k}_0 = \{ \left(
\begin{array}{ccc}
A & 0 \\
0 & B \\
\end{array}
\right): A \in \frak{so}(2), B \in \frak{so}(m)\} \cong \frak{so}(2) \oplus \frak{so}(m)$, and it is the Lie subalgebra of $\frak{g}_0$ corresponding to the connected Lie subgroup $K$ of $G$. 
Note that $G/K$ is an irreducible Hermitian symmetric space of non-compact type.

   The complexification of $\frak{g}_0$ is $\frak{g} = \frak{so}(m+2, \mathbb{C})$, and
\[ \frak{g} = 
\begin{cases}
\frak{b}_l & \textrm{if } m = 2l-1, \\
\frak{\delta}_l & \textrm{if } m = 2l-2. 
\end{cases}
\]

Let $\frak{k} = \frak{k}_0^\mathbb{C}\subset \frak{g}, \frak{p} = \frak{p}_0^\mathbb{C} \subset \frak{g},$  and 
$\frak{t}'_0$ be a maximal abelian subspace of $\frak{so}(m)$. Then $\frak{t}_0 = \frak{so}(2) \oplus \frak{t}'_0$ be a maximal abelian subspace of $\frak{k}_0$, and $\frak{h} = 
\frak{t}_0^\mathbb{C}$ is a Cartan subalgebra of $\frak{k}$ as well as of $\frak{g}$. Let $\Delta = \Delta(\frak{g}, \frak{h})$ be the set of all non-zero roots of $\frak{g}$ 
with respect to the Cartan subalgebra $\frak{h},$ similarly $\Delta_\frak{k} = \Delta(\frak{k}, \frak{h})$ be the set of all non-zero roots of $\frak{k}$ 
with respect to $\frak{h},$ and 
$\Delta_n = \Delta \setminus \Delta_\frak{k}=$ the set of all non-compact roots of $\frak{g}$ with respect to $\frak{h}$. 
Then $\frak{k}= \frak{h} + \sum_{\alpha \in \Delta_\frak{k}} \frak{g}^\alpha, \frak{p} = \sum_{\alpha \in \Delta_n} \frak{g}^\alpha,$ 
where $\frak{g}^\alpha$ is the root subspace of $\frak{g}$ of the root $\alpha \in \Delta$. Let $B$ denote the Killing form of $\frak{g}$. 
For any linear function  $\lambda$ on $\frak{h},$ there exists unique $H_\lambda \in \frak{h}$ such that 
\[ \lambda (H) = B (H, H_\lambda ) \textrm{ for all } H \in \frak{h} . \]
  Put $\langle \lambda , \mu \rangle = B(H_\lambda, H_\mu)$ for any linear functions $\lambda, \mu$ on $\frak{h}$, 
$H_\alpha ^* = 2 H_\alpha /\alpha (H_\alpha)$ for all $\alpha \in \Delta,$ and $\frak{h}_\mathbb{R} = \sum_{\alpha \in \Delta} \mathbb{R} H_\alpha$. Then $\frak{h}_\mathbb{R} = 
i\frak{t}_0.$ 

For $m \neq 2,$  let $\Delta^+$ be a Borel-de Siebenthal positive root system of $\Delta$ with a unique non-compact simple root $\phi_1$, that is 
\[n_{\phi_1}(\alpha) = 
\begin{cases}
0 & \textrm{if } \alpha \in \Delta_\frak{k}, \\
\pm 1 & \textrm{if } \alpha \in \Delta_n. 
\end{cases}
\] If $m=2,$ then $\Delta^+ = \{\phi_1, \phi_2\},$ 
where both of $\phi_1,$ and $\phi_2$ are non-compact and simple. 
Let $\Delta_\frak{k}^+ = \Delta^+ \cap \Delta_\frak{k}, 
\Delta_n^+ = \Delta^+ \cap \Delta_n,$ and $\Delta_n^- = -\Delta_n^+$. Write $\frak{p}_+ = \sum_{\alpha \in \Delta_n^+} \frak{g}^\alpha,$ and 
$\frak{p}_- = \sum_{\alpha \in \Delta_n^-} \frak{g}^\alpha$. Then $\frak{p} = \frak{p}_+ \oplus \frak{p}_-$ is the irreducible decomposition of $\frak{p}$ under the adjoint representation of 
$\frak{k},$ if $m \neq 2$. For $m \neq 2,$ let $\Phi_\frak{k} = \{ \phi_2, \phi_3, \ldots , \phi_l\}$ be the set of all simple roots in in $\Delta_\frak{k}^+$. 
Then $\Phi = \{\phi_1, \phi_2, \ldots , \phi_l\}$ is the set of all simple roots in $\Delta$. In the diagrams of this article, the non-compact roots are represented by black vertices.

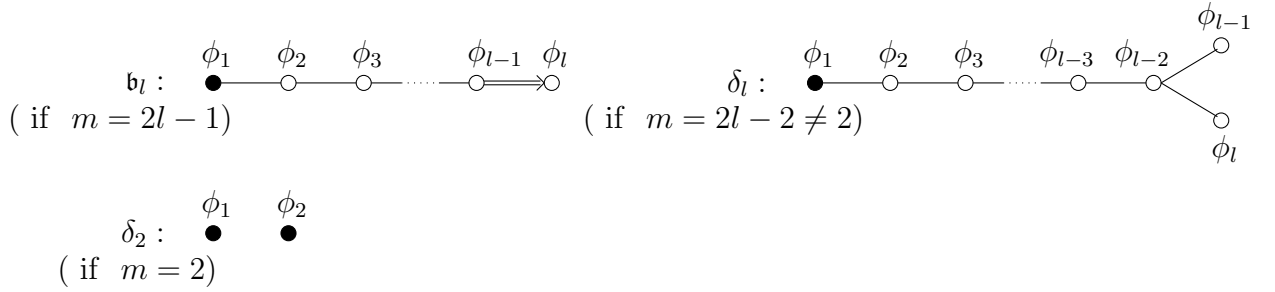
\begin{figure}[!h]
\begin{center} 
\caption{Dynkin diagram of $\frak{g}= \frak{so}(m+2, \mathbb{C})$}\label{dynkin} 
\begin{tikzpicture}

\filldraw [black] (0,0) circle [radius = 0.1]; 
\draw (1,0) circle [radius = 0.1]; 
\draw (2,0) circle [radius = 0.1]; 
\draw (3.5,0) circle [radius = 0.1]; 
\draw (4.5,0) circle [radius = 0.1]; 
\node [above] at (0.05,0.05) {$\phi_1$}; 
\node [above] at (1.05,0.05) {$\phi_2$}; 
\node [above] at (2.05,0.05) {$\phi_3$}; 
 \node [above] at (3.75,0.05) {$\phi_{l-1}$}; 
\node [above] at (4.55,0.05) {$\phi_l$}; 
\node [left] at (-0.5,0) {$\frak{b}_l :$}; 
\node[ left] at (0.4,-0.5) {(\textrm{ if } $m=2l-1$)}; 
\draw (0.1,0) -- (0.9,0); 
\draw (1.1,0) -- (1.9,0);
\draw (2.1,0) -- (2.5,0); 
\draw [dotted] (2.5,0) -- (3,0); 
\draw (3,0) -- (3.4,0); 
\draw (4.4,0) -- (4.3,0.1); 
\draw (4.4,0) -- (4.3,-0.1); 
\draw (3.6,0.025) -- (4.35,0.025); 
\draw (3.6,-0.025) -- (4.35,-0.025);

\filldraw [black] (8,0) circle [radius = 0.1]; 
\draw (9,0) circle [radius = 0.1]; 
\draw (10,0) circle [radius = 0.1]; 
\draw (11.5,0) circle [radius = 0.1]; 
\draw (12.5,0) circle [radius = 0.1]; 
\draw (13.4,0.5) circle [radius = 0.1]; 
\draw (13.4,-0.5) circle [radius = 0.1]; 
\node [above] at (8.05,0.05) {$\phi_1$}; 
\node [above] at (9.05,0.05) {$\phi_2$}; 
\node [above] at (10.05,0.05) {$\phi_3$}; 
\node [above] at (11.35,0.05) {$\phi_{l-3}$}; 
\node [above] at (12.35,0.05) {$\phi_{l-2}$}; 
\node [above] at (13.45,0.55) {$\phi_{l-1}$}; 
\node [below] at (13.45,-0.55) {$\phi_l$}; 
\node [left] at (7.5,0) {$\frak{\delta}_l :$}; 
\node [left] at (8.8,-0.5) {(\textrm{ if } $m = 2l-2 \neq 2$)}; 
\draw (8.1,0) -- (8.9,0); 
\draw (9.1,0) -- (9.9,0); 
\draw (10.1,0) -- (10.5,0); 
\draw [dotted] (10.5,0) -- (11,0); 
\draw (11,0) -- (11.4,0); 
\draw (11.6,0) -- (12.4,0); 
\draw (12.6,0) -- (13.35,0.45); 
\draw (12.6,0) -- (13.35,-0.45); 

\filldraw [black] (0,-2) circle [radius = 0.1]; 
\filldraw [black] (1,-2) circle [radius = 0.1]; 
\node [above] at (0.05,-1.95) {$\phi_1$}; 
\node [above] at (1.05,-1.95) {$\phi_2$}; 
\node [left] at (-0.5,-2) {$\frak{\delta}_2 :$}; 
\node[ left] at (0.2,-2.5) {(\textrm{ if } $m=2$)}; 

\end{tikzpicture} 
\end{center}
\end{figure} 

   Since $A_{\textrm{Ad}(k)(\frak{q})}$ is unitarily equivalent to $A_\frak{q}$ for all $k\in K,$ to determine all unitarily inequivalent 
$A_\frak{q},$ it is sufficient to determine all $\theta$-stable parabolic subalgebras $\frak{q}$ of $\frak{g}_0$ which contain 
$\frak{h} \oplus \sum_{\alpha \in \Delta_\frak{k}^+} \frak{g}^\alpha$. 
Let $\frak{q}$ be a $\theta$-stable parabolic subalgebra of $\frak{g}_0$ containing $\frak{h} \oplus \sum_{\alpha \in \Delta_\frak{k}^+} \frak{g}^\alpha$.  
Then there exists $x \in \frak{h}_\mathbb{R}$ such that 
$\frak{q} = \frak{q}_x = \frak{h} \oplus \sum_{\alpha(x) \ge 0, \alpha \in \Delta} \frak{g}^\alpha = \frak{l}_x \oplus \frak{u}_x,$ where 
$\frak{l}_x = \frak{h} \oplus \sum_{\alpha(x) = 0, \alpha \in \Delta} \frak{g}^\alpha$ is the Levi subalgebra of $\frak{q}_x,$ and 
$\frak{u}_x= \sum_{\alpha(x) > 0, \alpha \in \Delta} \frak{g}^\alpha$ is the nilradical of $\frak{q}_x$. Note that $\alpha (x) \ge 0$ for all $\alpha \in \Delta_\frak{k}^+.$ 
Write $\Delta(\frak{u}_x \cap \frak{p}_+) = \{ \beta \in \Delta_n^+ : \beta (x) > 0 \},$ and  $\Delta(\frak{u}_x \cap \frak{p}_-) = \{ \beta \in \Delta_n^- : \beta (x) > 0 \}$. 
For $x, y \in \frak{h}_\mathbb{R}, A_{\frak{q}_x}$ is unitarily equivalent to $A_{\frak{q}_x}$ {\it iff} $\Delta(\frak{u}_x \cap \frak{p}_+) \cup \Delta(\frak{u}_x \cap \frak{p}_-) = 
\Delta(\frak{u}_y \cap \frak{p}_+) \cup \Delta(\frak{u}_y \cap \frak{p}_-)$. So we will determine all possible candidates of $\Delta(\frak{u}_x \cap \frak{p}_+) \cup 
\Delta(\frak{u}_x \cap \frak{p}_-),$ where $x  \in \frak{h}_\mathbb{R}$ with $\alpha (x) \ge 0$ for all $\alpha \in \Delta_\frak{k}^+$. For $x \in \frak{h}_\mathbb{R}$ 
with $\alpha (x) \ge 0$ for all $\alpha \in \Delta_\frak{k}^+,$ we may 
write $x = H_\lambda$ for some linear function $\lambda$ on $\frak{h}_\mathbb{R}$ with $\langle \lambda, \alpha \rangle \ge 0$ for all $\alpha \in \Delta_\frak{k}^+$. 
We write $\frak{q}_\lambda = \frak{q}_x, \frak{l}_\lambda = \frak{l}_x,$ and $\frak{u}_\lambda = \frak{u}_x$. Thus 
$\Delta(\frak{u}_\lambda \cap \frak{p}_+) = \{ \beta \in \Delta_n^+ : \langle \lambda , \beta \rangle > 0 \},$ and  
$\Delta(\frak{u}_\lambda \cap \frak{p}_-) = \{ \beta \in \Delta_n^- : \langle \lambda, \beta \rangle > 0 \}$. Clearly $\Delta(\frak{u}_\lambda \cap \frak{p}_+)$ and 
$-\Delta(\frak{u}_\lambda \cap \frak{p}_-)$ are disjoint subsets of $\Delta_n^+$.

Now we begin our proofs with this elementary lemma. 

\begin{lemma}\cite[Remark 3.3(iii)]{mondal-sankaran2}\label{lemma}
Let $\lambda$ be a linear function on $\frak{h}_\mathbb{R}$ such that $\langle \lambda, \alpha \rangle \ge 0$ for all $\alpha \in \Delta_\frak{k}^+$. \\
(i) Let $\beta, \gamma \in \Delta_n$ be such that $\gamma > \beta,$ and both belong to $\Delta_n^+$ or $\Delta_n^-$. Then $\langle \lambda , \beta \rangle > 0 
\implies \langle \lambda , \gamma \rangle >0$.  \\ 
(ii) Let $\phi \in \Delta_\frak{k}^+$ be simple and $\beta \in \Delta_n$.
 If $\beta -\phi \in \Delta, \langle \lambda , \beta - \phi \rangle = 0,$ and $\langle \lambda , \beta \rangle > 0,$ then $\langle \lambda , \phi \rangle > 0$. \\ 
(iii) Let $\phi \in \Delta_\frak{k}^+$ be simple and $\beta \in \Delta_n$.
 If $\beta + \phi \in \Delta, \langle \lambda , \beta + \phi \rangle = 0,$ and $\langle \lambda , \beta \rangle = 0,$ then $\langle \lambda , \phi \rangle = 0$. 
 If $\beta - \phi \in \Delta, \langle \lambda , \beta - \phi \rangle = 0,$ and $\langle \lambda , \beta \rangle = 0,$ then $\langle \lambda , \phi \rangle = 0$. 
\end{lemma}

\begin{proof} 
(i) Let $\beta, \gamma \in \Delta_n$ be such that $\gamma > \beta$ and both belong to $\Delta_n^+$ or $\Delta_n^-$. 
Then $\gamma = \beta + \sum_{2\le i \le l} n_i \phi_i,$ where $n_i \in \mathbb{N} \cup \{0\}$ for all $2 \le i \le l$. 
Since $\langle \lambda , \beta \rangle > 0,$ and $\langle \lambda , \phi_i \rangle \ge  0$ for all $2 \le i \le l,$ we have $\langle \lambda , \gamma \rangle > 0$.  

(ii) $\langle \lambda , \beta - \phi \rangle = 0 \implies  \langle \lambda , \phi \rangle =  \langle \lambda , \beta \rangle > 0$. 

(iii) $\langle \lambda , \beta + \phi \rangle = 0 \implies  \langle \lambda , \phi \rangle = - \langle \lambda , \beta \rangle = 0,$ and \\ 
$\langle \lambda , \beta - \phi \rangle = 0 \implies  \langle \lambda , \phi \rangle =  \langle \lambda , \beta \rangle = 0$. 
\end{proof}

  Lemma \ref{lemma}(i) says that $\Delta(\frak{u}_\lambda \cap \frak{p}_+)$ is either empty or a set of the form $\cup_{1 \le i \le r }\{ \beta \in \Delta_n^+ : \beta \ge \xi_i \},$ and  
$\Delta(\frak{u}_\lambda \cap \frak{p}_-)$ is either empty or a set of the form 
$\cup_{1\le j \le s} \{ -\beta \in \Delta_n^- : -\beta \ge -\eta_j \}= \cup_{1\le j \le s} (-\{ \beta \in \Delta_n^+ : \beta \le \eta_j \})$, where 
$\{ \xi_1, \xi_2, \ldots , \xi_r\}, \{ \eta_1, \eta_2, \ldots , \eta_s\}$ are sets of pairwise non-comparable roots in $\Delta_n^+$. 

  If $\frak{g} = \frak{b}_l (l \ge 2)$, then $\Delta_n^+ = \{\phi_1, \phi_1 + \phi_2, \ldots , \phi_1 + \phi_2 + \cdots + \phi_l,  \phi_1 + \phi_2 + \cdots + 2\phi_l, 
\phi_1 + \phi_2 + \cdots + 2\phi_{l-1} + 2\phi_l, \ldots ,  \phi_1 + 2\phi_2 + \cdots + 2\phi_l \}$.  If $\frak{g} = \frak{b}_1$, then $\Delta_n^+ = \{\phi_1 \}$. 
If $\frak{g} = \frak{\delta}_l (l\ge 4)$, then 
$\Delta_n^+ = \{\phi_1, \phi_1 + \phi_2, \ldots , \phi_1 + \phi_2 + \cdots + \phi_{l-2}, \phi_1 + \phi_2 + \cdots + \phi_{l-2} + \phi_{l-1}, \phi_1 + \phi_2 + \cdots + \phi_{l-2} + \phi_l, 
 \phi_1 + \phi_2 + \cdots + \phi_{l-2} + \phi_{l-1} + \phi_l, \phi_1 + \phi_2 + \cdots + 2\phi_{l-2} + \phi_{l-1} + \phi_l, \ldots , \phi_1 + 2\phi_2 + \cdots + 2\phi_{l-2} + \phi_{l-1} + \phi_l \}$.
If $\frak{g} = \frak{\delta}_2$,  then $\Delta_n^+ = \{ \phi_1, \phi_2 \}$. 
If $\frak{g} = \frak{\delta}_3$,   then $\Delta_n^+ = \{\phi_1, \phi_1+\phi_2, \phi_1+\phi_3, \phi_1+\phi_2+\phi_3\}$. 
 
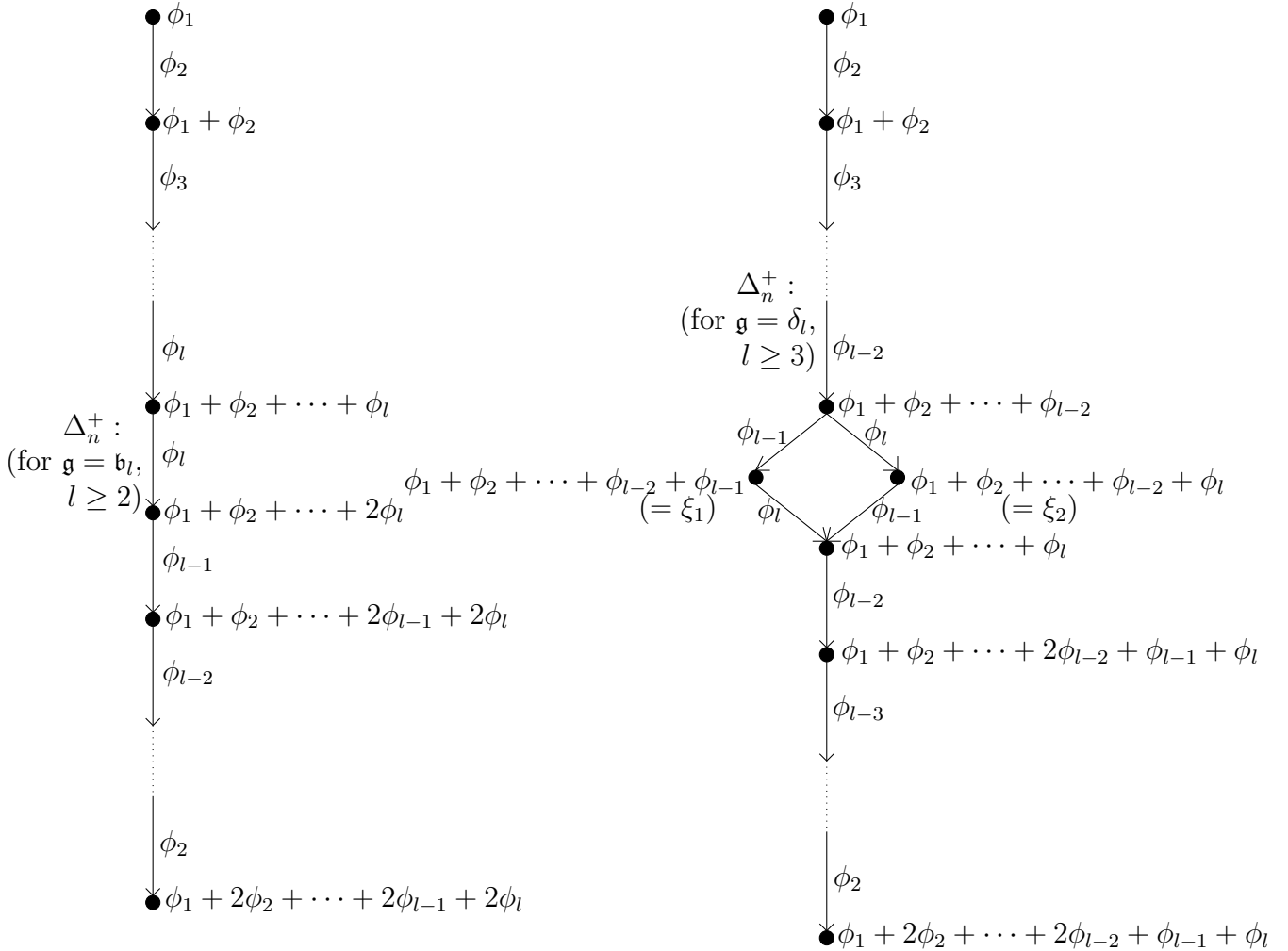
\begin{figure}[!h]
\begin{center} 
\caption{Diagram of $\Delta_n^+$}\label{diagram}
\begin{tikzpicture}

\filldraw [black] (0,0) circle [radius = 0.1]; 
\filldraw [black] (0,-1.5) circle [radius = 0.1]; 
\filldraw [black] (0,-5.5) circle [radius = 0.1]; 
\filldraw [black] (0,-7) circle [radius = 0.1];
\filldraw [black] (0,-8.5) circle [radius = 0.1];
\filldraw [black] (0,-12.5) circle [radius = 0.1];

\node [above] at (0.4,-0.3) {$\phi_1$}; 
\node [above] at (0.8,-1.8) {$\phi_1+\phi_2$}; 
\node [above] at (1.75,-5.8) {$\phi_1+\phi_2+\cdots+\phi_l$}; 
\node [above] at (1.85,-7.3) {$\phi_1+\phi_2+\cdots+2\phi_l$}; 
\node [above] at (2.6,-8.8) {$\phi_1+\phi_2+\cdots+2\phi_{l-1}+2\phi_l$}; 
\node [above] at (2.7,-12.8) {$\phi_1+2\phi_2+\cdots+2\phi_{l-1}+2\phi_l$}; 

\node [left] at (-0.3,-5.8) {$\Delta_n^+ :$}; 
\node[ left] at (0,-6.3) {(\textrm{for} $\frak{g}=\frak{b}_l$,};
\node[ left] at (0,-6.8) {$l \ge 2$)}; 

\node [above] at (0.3,-1) {$\phi_2$}; 
\node [above] at (0.3,-2.6) {$\phi_3$};
\node [above] at (0.3,-5) {$\phi_l$};
\node [above] at (0.3,-6.5) {$\phi_l$};
\node [above] at (0.5,-8) {$\phi_{l-1}$};
\node [above] at (0.5,-9.6) {$\phi_{l-2}$};
\node [above] at (0.3,-12) {$\phi_2$};

\draw (0,-0.1) -- (0,-1.4); 
\draw (0,-1.6) -- (0,-3);
\draw [dotted] (0,-3) -- (0,-4); 
\draw (0,-4) -- (0,-5.4); 
\draw (0,-5.6) -- (0,-6.9); 
\draw (0,-7.1) -- (0,-8.4); 
\draw (0,-8.6) -- (0,-10);
\draw [dotted] (0,-10) -- (0,-11); 
\draw (0,-11) -- (0,-12.4);

\draw (0,-1.4) -- (-0.1,-1.3); 
\draw (0,-1.4) -- (0.1,-1.3); 
\draw (0,-3) -- (-0.1,-2.9); 
\draw (0,-3) -- (0.1,-2.9); 
\draw (0,-5.4) -- (-0.1,-5.3); 
\draw (0,-5.4) -- (0.1,-5.3); 
\draw (0,-6.9) -- (-0.1,-6.8); 
\draw (0,-6.9) -- (0.1,-6.8); 
\draw (0,-8.4) -- (-0.1,-8.3); 
\draw (0,-8.4) -- (0.1,-8.3); 
\draw (0,-10) -- (-0.1,-9.9); 
\draw (0,-10) -- (0.1,-9.9); 
\draw (0,-12.4) -- (-0.1,-12.3); 
\draw (0,-12.4) -- (0.1,-12.3);

\filldraw [black] (9.5,0) circle [radius = 0.1]; 
\filldraw [black] (9.5,-1.5) circle [radius = 0.1]; 
\filldraw [black] (9.5,-5.5) circle [radius = 0.1]; 
\filldraw [black] (8.5,-6.5) circle [radius = 0.1];
\filldraw [black] (10.5,-6.5) circle [radius = 0.1];
\filldraw [black] (9.5,-7.5) circle [radius = 0.1];
\filldraw [black] (9.5,-9) circle [radius = 0.1];
\filldraw [black] (9.5,-13) circle [radius = 0.1];

\node [above] at (9.9,-0.3) {$\phi_1$}; 
\node [above] at (10.3,-1.8) {$\phi_1+\phi_2$}; 
\node [above] at (11.45,-5.8) {$\phi_1+\phi_2+\cdots+\phi_{l-2}$}; 
\node [above] at (5.95,-6.9) {$\phi_1+\phi_2+\cdots+\phi_{l-2}+\phi_{l-1}$}; 
\node [above] at (7.4,-7.3) {$(=\xi_1)$};
\node [above] at (12.9,-6.9) {$\phi_1+\phi_2+\cdots+\phi_{l-2}+\phi_l$}; 
\node [above] at (12.5,-7.3) {$(=\xi_2)$};
\node [above] at (11.3,-7.8) {$\phi_1+\phi_2 +\cdots+\phi_l$}; 
\node [above] at (12.65,-9.3) {$\phi_1+\phi_2+\cdots+2\phi_{l-2}+\phi_{l-1}+\phi_l$};
\node [above] at (12.7,-13.3) {$\phi_1+2\phi_2+\cdots+2\phi_{l-2}+\phi_{l-1}+\phi_l$}; 

\node [left] at (9.2,-3.8) {$\Delta_n^+ :$}; 
\node[ left] at (9.5,-4.3) {(\textrm{for} $\frak{g}=\frak{\delta}_l$,}; 
\node[ left] at (9.5,-4.8) {$l \ge 3$)}; 

\node [above] at (9.8,-1) {$\phi_2$}; 
\node [above] at (9.8,-2.6) {$\phi_3$};
\node [above] at (9.95,-5) {$\phi_{l-2}$};
\node [above] at (8.6,-6.2) {$\phi_{l-1}$};
\node [above] at (10.2,-6.2) {$\phi_l$};
\node [above] at (8.7,-7.3) {$\phi_l$};
\node [above] at (10.5,-7.3) {$\phi_{l-1}$};
\node [above] at (9.95,-8.5) {$\phi_{l-2}$};
\node [above] at (9.95,-10.1) {$\phi_{l-3}$};
\node [above] at (9.8,-12.5) {$\phi_2$};

\draw (9.5,-0.1) -- (9.5,-1.4); 
\draw (9.5,-1.6) -- (9.5,-3);
\draw [dotted] (9.5,-3) -- (9.5,-4); 
\draw (9.5,-4) -- (9.5,-5.4); 
\draw (9.5,-5.6) -- (8.5,-6.4); 
\draw (9.5,-5.6) -- (10.5,-6.4); 
\draw (8.5,-6.6) -- (9.5,-7.4); 
\draw (10.5,-6.6) -- (9.5,-7.4); 
\draw (9.5,-7.6) -- (9.5,-8.9); 
\draw (9.5,-9.1) -- (9.5,-10.5);
\draw [dotted] (9.5,-10.5) -- (9.5,-11.5); 
\draw (9.5,-11.5) -- (9.5,-12.9);

\draw (9.5,-1.4) -- (9.4,-1.3); 
\draw (9.5,-1.4) -- (9.6,-1.3); 
\draw (9.5,-3) -- (9.4,-2.9); 
\draw (9.5,-3) -- (9.6,-2.9); 
\draw (9.5,-5.4) -- (9.4,-5.3); 
\draw (9.5,-5.4) -- (9.6,-5.3); 
\draw (8.5,-6.4) -- (8.6,-6.2); 
\draw (8.5,-6.4) -- (8.7,-6.4); 
\draw (10.5,-6.4) -- (10.5,-6.2); 
\draw (10.5,-6.4) -- (10.3,-6.4); 
\draw (9.3,-7.4) -- (9.5,-7.4); 
\draw (9.45,-7.2) -- (9.5,-7.4); 
\draw (9.7,-7.4) -- (9.5,-7.4); 
\draw (9.55,-7.2) -- (9.5,-7.4); 
\draw (9.5,-8.9) -- (9.4,-8.8); 
\draw (9.5,-8.9) -- (9.6,-8.8); 
\draw (9.5,-10.5) -- (9.4,-10.4); 
\draw (9.5,-10.5) -- (9.6,-10.4); 
\draw (9.5,-12.9) -- (9.4,-12.8); 
\draw (9.5,-12.9) -- (9.6,-12.8); 

\end{tikzpicture} 
\end{center}
\end{figure} 

\begin{tikzpicture}

\filldraw [black] (0,0) circle [radius = 0.1]; 
\filldraw [black] (1.5,0) circle [radius = 0.1]; 

\node [below] at (0,-0.05) {$\phi_1$}; 
\node [below] at (1.5,-0.05) {$\phi_2$}; 

\node [left] at (-0.3,0) {$\Delta_n^+ :$}; 
\node[ left] at (-0.3,-0.5) {(\textrm{for} $\frak{g}=\frak{\delta}_2$)};

\end{tikzpicture} 

   In the Figure \ref{diagram}, the vertices represent roots in $\Delta_n^+$. Two roots $\beta, \gamma \in \Delta_n^+$ are joined by a line with an arrow in the direction of 
$\gamma$ if $\gamma = \beta + \phi$ for some simple root $\phi \in \Delta_\frak{k}^+$. In this case, the simple root $\phi$ is given on one side of the
 line.

\subsection{Proof of Th.\ref{th1}}\label{proof}
Let $\omega_1, \omega_2, \ldots, \omega_l$ be the fundamental weights of $\frak{g}$ corresponding to the simple roots $\phi_1, \phi_2, \ldots, \phi_l$ respectively. 

(i) {\bf $\frak{g} = \frak{b}_l(l>1):$} Lemma \ref{lemma}(i) and the diagram of $\Delta_n^+$ in Figure \ref{diagram} show that $\Delta(\frak{u}_\lambda \cap \frak{p}_+)$ is either empty or 
a set of the form $\{ \beta \in \Delta_n^+ : \beta \ge \xi\},$ and $\Delta(\frak{u}_\lambda \cap \frak{p}_-)$ is either empty or a set of the form 
$\{ -\beta \in \Delta_n^- : -\beta \ge -\eta \}= -\{ \beta \in \Delta_n^+ : \beta \le \eta \}$, and $\Delta(\frak{u}_\lambda \cap \frak{p}_+), -\Delta(\frak{u}_\lambda \cap \frak{p}_-)$ are 
disjoint subsets of $\Delta_n^+$, where $\xi, \eta \in \Delta_n^+.$ \\ 
Let $\Delta(\frak{u}_\lambda \cap \frak{p}_-)$ be empty. Then $\Delta(\frak{u}_\lambda \cap \frak{p}_+)=\{ \beta \in \Delta_n^+ : \beta \ge \xi\},$ where $\xi > \phi_1 + \phi_2 + \cdots +
\phi_l$ is not possible. For then $\xi = \phi_1 + \cdots + \phi_{i-1} + 2\phi_i + \cdots + 2\phi_l,$ where $2 \le i \le l$. So $\langle \lambda , \phi_i \rangle > 0,$ by Lemma \ref{lemma}(ii). 
Again $\langle \lambda , \phi_1+\phi_2+\cdots+\phi_i \rangle =0, \langle \lambda , \phi_1+\phi_2+\cdots+\phi_{i-1} \rangle =0.$ Thus $\langle \lambda , \phi_i \rangle =0,$ a contradiction. 
If $\xi \le \phi_1 + \phi_2 + \cdots +\phi_l,$ then $\xi = \phi_1+\cdots+\phi_i,$ for some $1 \le i \le l,$ and $\Delta(\frak{u}_\lambda \cap \frak{p}_+)=\{ \beta \in \Delta_n^+ : \beta \ge \xi\}, 
\Delta(\frak{u}_\lambda \cap \frak{p}_-)=\phi,$ where $\lambda= \omega_i.$ Also $\Delta(\frak{u}_\lambda \cap \frak{p}_+)=\phi, \Delta(\frak{u}_\lambda \cap \frak{p}_-)=\phi,$ for 
$\lambda =0.$ Thus the number of equivalence classes of irreducible unitary representations with non-zero $(\frak{g}, K)$-cohomology for which 
$\Delta(\frak{u}_\lambda \cap \frak{p}_-) = \phi,$ is $l+1$.  \\
Let $\Delta(\frak{u}_\lambda \cap \frak{p}_-)=-\{ \beta \in \Delta_n^+ : \beta \le \phi_1 +\cdots + \phi_i \},$ where $1 \le i \le l-1.$ Then 
$\Delta(\frak{u}_\lambda \cap \frak{p}_+)=\{ \beta \in \Delta_n^+ : \beta \ge \xi\},$ where $\xi > \phi_1+\cdots+\phi_l, \xi \neq \phi_1+\cdots+\phi_i+2\phi_{i+1}+\cdots+2\phi_l$ is not possible. 
Because  $\xi > \phi_1+\cdots+\phi_l, \xi \neq \phi_1+\cdots+\phi_i+2\phi_{i+1}+\cdots+2\phi_l$ implies 
$\xi=\phi_1+\cdots+\phi_{j-1}+2\phi_j+\cdots+2\phi_l$ for some $2 \le j \le l, j \neq i+1.$ Then $\langle \lambda , \phi_{i+1} \rangle > 0, \langle \lambda , \phi_j \rangle > 0,$ by Lemma \ref{lemma}(ii). 
If $2 \le j \le i,$ then $\langle \lambda , \phi_1+\cdots+\phi_i+2\phi_{i+1}+\cdots+2\phi_l \rangle =0, \langle \lambda , \phi_1+\cdots+\phi_{i+1}+2\phi_{i+2}+\cdots+2\phi_l \rangle =0$. Thus 
$\langle \lambda , \phi_{i+1} \rangle =0$, a contradiction.  
If $i+2\le j \le l,$ then $\langle \lambda , \phi_1+\phi_2+\cdots+\phi_j \rangle =0, \langle \lambda , \phi_1+\phi_2+\cdots+\phi_{j-1} \rangle =0$. 
Thus $\langle \lambda , \phi_j \rangle =0$, a contradiction. If $\xi \le \phi_1 + \phi_2 + \cdots +\phi_l,$ then $\xi = \phi_1+\cdots+\phi_j,$ for some $i+1 \le j \le l,$ and 
$\Delta(\frak{u}_\lambda \cap \frak{p}_+)=\{ \beta \in \Delta_n^+ : \beta \ge \xi\},\Delta(\frak{u}_\lambda \cap \frak{p}_-)=-\{ \beta \in \Delta_n^+ : \beta \le \phi_1 +\cdots + \phi_i \},$ 
where $\lambda= \frac{\omega_{i+1}}{\langle \omega_{i+1},\phi_{i+1}\rangle}+\frac{\omega_j}{\langle \omega_j,\phi_j\rangle}-\frac{\omega_1}{\langle \omega_1,\phi_1\rangle}.$
If $\xi = \phi_1+\cdots+\phi_i+2\phi_{i+1}+\cdots+2\phi_l,$ then 
$\Delta(\frak{u}_\lambda \cap \frak{p}_+)=\{ \beta \in \Delta_n^+ : \beta \ge \xi\},\Delta(\frak{u}_\lambda \cap \frak{p}_-)=-\{ \beta \in \Delta_n^+ : \beta \le \phi_1 +\cdots + \phi_i \},$ 
for $\lambda= \frac{\omega_{i+1}}{\langle \omega_{i+1},\phi_{i+1}\rangle}-\frac{\omega_1}{\langle \omega_1,\phi_1\rangle}.$
Also $\Delta(\frak{u}_\lambda \cap \frak{p}_+)=\phi$ is not possible, for $\langle \lambda , \phi_{i+1} \rangle > 0,$ by Lemma \ref{lemma}(ii); and 
$\langle \lambda , \phi_1+\cdots+\phi_i+2\phi_{i+1}+\cdots+2\phi_l \rangle =0, \langle \lambda , \phi_1+\cdots+\phi_{i+1}+2\phi_{i+2}+\cdots+2\phi_l \rangle =0$. Thus $\langle \lambda , \phi_{i+1} \rangle =0$, 
a contradiction. Hence the number of equivalence classes of irreducible unitary representations with non-zero $(\frak{g}, K)$-cohomology for which 
$\Delta(\frak{u}_\lambda \cap \frak{p}_-) = -\{ \beta \in \Delta_n^+ : \beta \le \phi_1 +\cdots + \phi_i \},$ is $l - i+1$ for all $1 \le i \le l-1.$   \\ 
Let $\Delta(\frak{u}_\lambda \cap \frak{p}_-)=-\{ \beta \in \Delta_n^+ : \beta \le \phi_1 +\cdots + \phi_l \}.$ Since $\xi > \phi_1+\cdots+\phi_l, 
\xi=\phi_1+\cdots+\phi_{j-1}+2\phi_j+\cdots+2\phi_l$ for some $2 \le j \le l.$ Now $\Delta(\frak{u}_\lambda \cap \frak{p}_+)=\{ \beta \in \Delta_n^+ : \beta \ge \xi\},
\Delta(\frak{u}_\lambda \cap \frak{p}_-)=-\{ \beta \in \Delta_n^+ : \beta \le \phi_1 +\cdots + \phi_l \},$ 
for $\lambda= \frac{\omega_l}{\langle \omega_l,\phi_l \rangle}+\frac{\omega_j}{\langle \omega_j,\phi_j \rangle}-\frac{3\omega_1}{\langle \omega_1,\phi_1\rangle}.$ 
Also $\Delta(\frak{u}_\lambda \cap \frak{p}_+)=\phi, \Delta(\frak{u}_\lambda \cap \frak{p}_-)=-\{ \beta \in \Delta_n^+ : \beta \le \phi_1 +\cdots + \phi_l \},$ for 
$\lambda = \frac{\omega_l}{\langle \omega_l,\phi_l \rangle} -\frac{2\omega_1}{\langle \omega_1,\phi_1\rangle}.$ 
Thus the number of equivalence classes of irreducible unitary representations with non-zero $(\frak{g}, K)$-cohomology for which 
$\Delta(\frak{u}_\lambda \cap \frak{p}_-) = -\{ \beta \in \Delta_n^+ : \beta \le \phi_1 +\cdots + \phi_l \},$ is $l.$      \\ 
Let $\Delta(\frak{u}_\lambda \cap \frak{p}_-)=-\{ \beta \in \Delta_n^+ : \beta \le \phi_1 +\cdots + \phi_{i-1}+2\phi_i + \cdots+2\phi_l \}, 2 \le i \le l.$ 
Then $\Delta(\frak{u}_\lambda \cap \frak{p}_+)=\phi, \Delta(\frak{u}_\lambda \cap \frak{p}_-)=-\{ \beta \in \Delta_n^+ : \beta \le \phi_1 +\cdots + \phi_{i-1}+2\phi_i + \cdots+2\phi_l  \},$ for 
$\lambda = \frac{\omega_{i-1}}{\langle \omega_{i-1},\phi_{i-1} \rangle} -\frac{2\omega_1}{\langle \omega_1,\phi_1\rangle}$ if $3 \le i \le l,$ and 
$\lambda = -\omega_1$ if $i=2.$ If $\xi=\phi_1+\cdots+\phi_{j-1}+2\phi_j+\cdots+2\phi_l$ for some $2 \le j < i,$ then $\Delta(\frak{u}_\lambda \cap \frak{p}_+)=\{ \beta \in \Delta_n^+ : \beta \ge \xi\},
\Delta(\frak{u}_\lambda \cap \frak{p}_-)=-\{ \beta \in \Delta_n^+ : \beta \le \phi_1 +\cdots + \phi_{i-1}+2\phi_i + \cdots+2\phi_l \},$  
for $\lambda= \frac{\omega_{i-1}}{\langle \omega_{i-1},\phi_{i-1} \rangle}+\frac{\omega_j}{\langle \omega_j,\phi_j \rangle}-\frac{3\omega_1}{\langle \omega_1,\phi_1\rangle}.$ 
Hence the number of equivalence classes of irreducible unitary representations with non-zero $(\frak{g}, K)$-cohomology for which 
$\Delta(\frak{u}_\lambda \cap \frak{p}_-) = -\{ \beta \in \Delta_n^+ : \beta \le \phi_1 +\cdots + \phi_{i-1}+2\phi_i + \cdots+2\phi_l \},$ is $i-1$ for all $2 \le i \le l.$   \\ 
Thus the number of equivalence classes of irreducible unitary representations with non-zero $(\frak{g}, K)$-cohomology of the Lie group $SO_0(2,2l-1)(l > 1)$ is 
$A=(l+1)+l+\cdots+2+l+1+\cdots+(l-1)=3l + (l-1)l/2 + (l-1)l/2 = l(l+2).$ 

{\bf $\frak{g} = \frak{b}_1:$} In this case $\Delta_n^+ =\{\phi_1\}.$ If $\lambda = 0,$ then $\Delta(\frak{u}_\lambda \cap \frak{p}_-), \Delta(\frak{u}_\lambda \cap \frak{p}_+)$ are 
empty. $\lambda = \omega_1$ implies $\Delta(\frak{u}_\lambda \cap \frak{p}_-)=\phi, \Delta(\frak{u}_\lambda \cap \frak{p}_+)=\{\phi_1\};$ and 
$\lambda = -\omega_1$ implies $\Delta(\frak{u}_\lambda \cap \frak{p}_+)=\phi, \Delta(\frak{u}_\lambda \cap \frak{p}_-)=\{-\phi_1\}.$ Thus $A=3=l(l+2).$   

{\bf $\frak{g} = \frak{\delta}_l(l \ge 3):$} Lemma \ref{lemma}(i) and the diagram of $\Delta_n^+$ in Figure \ref{diagram} show that $\Delta(\frak{u}_\lambda \cap \frak{p}_+)$ is either empty or 
a set of the form $\{ \beta \in \Delta_n^+ : \beta \ge \xi\},$ or $\{ \beta \in \Delta_n^+ : \beta \ge \xi_1 \textrm{ or } \xi_2 \},$ and $\Delta(\frak{u}_\lambda \cap \frak{p}_-)$ is either empty or 
a set of the form $\{ -\beta \in \Delta_n^- : -\beta \ge -\eta \}= -\{ \beta \in \Delta_n^+ : \beta \le \eta \},$ or $-\{ \beta \in \Delta_n^+ : \beta \le \xi_1 \textrm{ or } \xi_2 \},$ where 
$\xi, \eta \in \Delta_n^+; \xi_1=\phi_1+\cdots+\phi_{l-2}+\phi_{l-1},\xi_2=\phi_1+\cdots+\phi_{l-2}+\phi_l.$ \\ 
Let $\Delta(\frak{u}_\lambda \cap \frak{p}_-)$ be empty. Then $\Delta(\frak{u}_\lambda \cap \frak{p}_+)=\{ \beta \in \Delta_n^+ : \beta \ge \xi\},$ where $\xi \ge \phi_1 + \phi_2 + \cdots +
\phi_l$ is not possible. For then $\xi = \phi_1+\cdots +\phi_l,$ or $\xi = \phi_1 + \cdots + \phi_{i-1} + 2\phi_i + \cdots + 2\phi_{l-2}+\phi_{l-1}+\phi_l,$ where $2 \le i \le l-2$. If $\xi = \phi_1+\cdots +\phi_l,$ then 
$\langle \lambda , \phi_{l-1} \rangle > 0, \langle \lambda , \phi_l \rangle > 0,$ by Lemma \ref{lemma}(ii). 
Again $\langle \lambda , \phi_1+\cdots+\phi_{l-1} \rangle =0, \langle \lambda , \phi_1+\cdots+\phi_{l-2}+\phi_l \rangle =0, \langle \lambda , \phi_1+\cdots+\phi_{l-2} \rangle =0$. 
Thus $\langle \lambda , \phi_{l-1} \rangle =0, \langle \lambda , \phi_l \rangle =0,$ a contradiction. If $\xi = \phi_1 + \cdots + \phi_{i-1} + 2\phi_i + \cdots + 2\phi_{l-2}+\phi_{l-1}+\phi_l (2 \le i \le l-2),$
then $\langle \lambda , \phi_i \rangle > 0,$ by Lemma \ref{lemma}(ii). 
Again $\langle \lambda , \phi_1+\phi_2+\cdots+\phi_i \rangle =0, \langle \lambda , \phi_1+\phi_2+\cdots+\phi_{i-1} \rangle =0$. Thus $\langle \lambda , \phi_i \rangle =0$, a contradiction. 
If $\xi < \phi_1 + \phi_2 + \cdots +\phi_l,$ then $\xi = \phi_1+\cdots+\phi_i,$ for some $1 \le i \le l-1,$ or $\xi = \xi_2$ and 
$\Delta(\frak{u}_\lambda \cap \frak{p}_+)=\{ \beta \in \Delta_n^+ : \beta \ge \xi\}, \Delta(\frak{u}_\lambda \cap \frak{p}_-)=\phi,$ for 
$\lambda= \omega_i, \omega_l$ respectively. Also $\Delta(\frak{u}_\lambda \cap \frak{p}_+)= \{ \beta \in \Delta_n^+ : \beta \ge \xi_1 \textrm{ or } \xi_2 \}, \Delta(\frak{u}_\lambda \cap \frak{p}_-)=\phi,$ 
for $\lambda= \omega_{l-1} +\omega_l,$ and 
$\Delta(\frak{u}_\lambda \cap \frak{p}_+)=\phi, \Delta(\frak{u}_\lambda \cap \frak{p}_-)=\phi,$ for 
$\lambda =0.$ Thus the number of equivalence classes of irreducible unitary representations with non-zero $(\frak{g}, K)$-cohomology for which 
$\Delta(\frak{u}_\lambda \cap \frak{p}_-) = \phi,$ is $l+2$.  \\
Let $\Delta(\frak{u}_\lambda \cap \frak{p}_-)=-\{ \beta \in \Delta_n^+ : \beta \le \phi_1 +\cdots + \phi_i \},$ where $1 \le i \le l-3, l \ge 4.$ Then 
$\Delta(\frak{u}_\lambda \cap \frak{p}_+)=\{ \beta \in \Delta_n^+ : \beta \ge \xi\},$ where $\xi \ge \phi_1+\cdots+\phi_l, \xi \neq \phi_1+\cdots+\phi_i+2\phi_{i+1}+\cdots+2\phi_{l-2}+\phi_{l-1}+\phi_l$ is not possible. 
For if $\xi = \phi_1+\cdots+\phi_l,$ then $\langle \lambda , \phi_{l-1} \rangle > 0, \langle \lambda , \phi_l \rangle > 0$ by Lemma \ref{lemma}(ii). Again 
$\langle \lambda , \phi_1+\phi_2+\cdots+\phi_{l-2} \rangle =0, \langle \lambda , \phi_1+\phi_2+\cdots+\phi_{l-1} \rangle =0, \langle \lambda , \phi_1+\phi_2+\cdots+\phi_{l-2}+\phi_l \rangle =0.$ Thus 
$\langle \lambda , \phi_{l-1} \rangle = 0, \langle \lambda , \phi_l \rangle = 0,$ a contradiction. 
Now  $\xi > \phi_1+\cdots+\phi_l, \xi \neq \phi_1+\cdots+\phi_i+2\phi_{i+1}+\cdots+2\phi_{l-2}+\phi_{l-1}+\phi_l$ implies 
$\xi=\phi_1+\cdots+\phi_{j-1}+2\phi_j+\cdots+2\phi_{l-2}+\phi_{l-1}+\phi_l$ for some $2 \le j \le l-2, j \neq i+1.$ Then $\langle \lambda , \phi_{i+1} \rangle > 0, \langle \lambda , \phi_j \rangle > 0,$ by Lemma \ref{lemma}(ii). 
If $2 \le j \le i,$ then $\langle \lambda , \phi_1+\cdots+\phi_i+2\phi_{i+1}+\cdots+2\phi_l \rangle =0, \langle \lambda , \phi_1+\cdots+\phi_{i+1}+2\phi_{i+2}+\cdots+2\phi_l \rangle =0$. Thus 
$\langle \lambda , \phi_{i+1} \rangle =0$, a contradiction.  
If $i+2\le j \le l-2,$ then $\langle \lambda , \phi_1+\phi_2+\cdots+\phi_j \rangle =0, \langle \lambda , \phi_1+\phi_2+\cdots+\phi_{j-1} \rangle =0$. 
Thus $\langle \lambda , \phi_j \rangle =0$, a contradiction. If $\xi < \phi_1 + \phi_2 + \cdots +\phi_l,$ then $\xi = \phi_1+\cdots+\phi_j,$ for some $i+1 \le j \le l-1,$ or $\xi = \xi_2.$
Now $\Delta(\frak{u}_\lambda \cap \frak{p}_+)=\{ \beta \in \Delta_n^+ : \beta \ge \phi_1+\cdots+\phi_j\}(i+1 \le j \le l-1),\Delta(\frak{u}_\lambda \cap \frak{p}_-)=-\{ \beta \in \Delta_n^+ : \beta \le \phi_1 +\cdots + \phi_i \}$ 
for $\lambda= \omega_{i+1}+\omega_j-\omega_1; \Delta(\frak{u}_\lambda \cap \frak{p}_+)=\{ \beta \in \Delta_n^+ : \beta \ge \xi_2 \}, 
\Delta(\frak{u}_\lambda \cap \frak{p}_-)=-\{ \beta \in \Delta_n^+ : \beta \le \phi_1 +\cdots + \phi_i \}$ 
for $\lambda= \omega_{i+1}+\omega_l-\omega_1;$ and $\Delta(\frak{u}_\lambda \cap \frak{p}_+)=\{ \beta \in \Delta_n^+ : \beta \ge \phi_1+\cdots+\phi_i+2\phi_{i+1}+\cdots+2\phi_{l-2}+\phi_{l-1}+\phi_l,\}, 
\Delta(\frak{u}_\lambda \cap \frak{p}_-)=-\{ \beta \in \Delta_n^+ : \beta \le \phi_1 +\cdots + \phi_i \}$ for $\lambda= \omega_{i+1}-\omega_1.$
Also $\Delta(\frak{u}_\lambda \cap \frak{p}_+)=\{ \beta \in \Delta_n^+ : \beta \ge \xi_1 \textrm{ or } \xi_2 \}, \Delta(\frak{u}_\lambda \cap \frak{p}_-)=-\{ \beta \in \Delta_n^+ : \beta \le \phi_1 +\cdots + \phi_i \}$ for 
$\lambda= \omega_{i+1}+\omega_{l-1}+\omega_l -\omega_1.$
Again $\Delta(\frak{u}_\lambda \cap \frak{p}_+)=\phi$ is not possible, for $\langle \lambda , \phi_{i+1} \rangle > 0,$ by Lemma \ref{lemma}(ii); and 
$\langle \lambda , \phi_1+\cdots+\phi_i+2\phi_{i+1}+\cdots+2\phi_l \rangle =0, \langle \lambda , \phi_1+\cdots+\phi_{i+1}+2\phi_{i+2}+\cdots+2\phi_l \rangle =0$. Thus $\langle \lambda , \phi_{i+1} \rangle =0$, 
a contradiction. Hence the number of equivalence classes of irreducible unitary representations with non-zero $(\frak{g}, K)$-cohomology for which 
$\Delta(\frak{u}_\lambda \cap \frak{p}_-) = -\{ \beta \in \Delta_n^+ : \beta \le \phi_1 +\cdots + \phi_i \},$ is $l - i+2$ for all $1 \le i \le l-3.$   \\ 
Let $\Delta(\frak{u}_\lambda \cap \frak{p}_-)=-\{ \beta \in \Delta_n^+ : \beta \le \phi_1 +\cdots + \phi_{l-2} \}.$ Then 
$\langle \lambda , \phi_{l-1} \rangle > 0, \langle \lambda , \phi_l \rangle > 0$ by Lemma \ref{lemma}(ii). Hence 
$\Delta(\frak{u}_\lambda \cap \frak{p}_+)= \phi,$ or $\{ \beta \in \Delta_n^+ : \beta \ge \xi\},$ where $\xi > \phi_1+\cdots+\phi_l$ is not possible. For then 
$\langle \lambda , \phi_1+\phi_2+\cdots+\phi_l \rangle =0, \langle \lambda , \phi_1+\phi_2+\cdots+\phi_{l-1} \rangle =0, \langle \lambda , \phi_1+\phi_2+\cdots+\phi_{l-2}+\phi_l \rangle =0.$ Thus 
$\langle \lambda , \phi_l \rangle = 0, \langle \lambda , \phi_{l-1} \rangle = 0,$ a contradiction. Now $\Delta(\frak{u}_\lambda \cap \frak{p}_+)=\{ \beta \in \Delta_n^+ : \beta \ge \xi_1 \}, 
\Delta(\frak{u}_\lambda \cap \frak{p}_-)=-\{ \beta \in \Delta_n^+ : \beta \le \phi_1 +\cdots + \phi_{l-2} \}$ for $\lambda= 2\omega_{l-1}+\omega_l-\omega_1; 
\Delta(\frak{u}_\lambda \cap \frak{p}_+)=\{ \beta \in \Delta_n^+ : \beta \ge \xi_2 \}, 
\Delta(\frak{u}_\lambda \cap \frak{p}_-)=-\{ \beta \in \Delta_n^+ : \beta \le \phi_1 +\cdots + \phi_{l-2} \}$ for $\lambda= \omega_{l-1}+2\omega_l-\omega_1;$ and 
$\Delta(\frak{u}_\lambda \cap \frak{p}_+)=\{ \beta \in \Delta_n^+ : \beta \ge \phi_1+\cdots+\phi_l\}, 
\Delta(\frak{u}_\lambda \cap \frak{p}_-)=-\{ \beta \in \Delta_n^+ : \beta \le \phi_1 +\cdots + \phi_{l-2} \}$ for $\lambda= \omega_{l-1}+\omega_l-\omega_1.$ 
Also $\Delta(\frak{u}_\lambda \cap \frak{p}_+)=\{ \beta \in \Delta_n^+ : \beta \ge \xi_1 \textrm{ or } \xi_2 \}, \Delta(\frak{u}_\lambda \cap \frak{p}_-)=-\{ \beta \in \Delta_n^+ : \beta \le \phi_1 +\cdots + \phi_{l-2} \}$ for 
$\lambda= 2\omega_{l-1}+2\omega_l-\omega_1.$ Hence the number of equivalence classes of irreducible unitary representations with non-zero $(\frak{g}, K)$-cohomology for which 
$\Delta(\frak{u}_\lambda \cap \frak{p}_-) = -\{ \beta \in \Delta_n^+ : \beta \le \phi_1 +\cdots + \phi_{l-2} \},$ is $4.$ \\
Let $\Delta(\frak{u}_\lambda \cap \frak{p}_-)=-\{ \beta \in \Delta_n^+ : \beta \le \phi_1+\cdots+\phi_{l-2}+\phi_a \},$ where $a=l-1,l;$ and $b=l \textrm{ if }a=l-1, \textrm{and } b=l-1 
\textrm{ if } a=l.$  
Then $\Delta(\frak{u}_\lambda \cap \frak{p}_+)=\{ \beta \in \Delta_n^+ : \beta \ge \phi_1+\cdots+\phi_{l-2}+\phi_b \}, 
\Delta(\frak{u}_\lambda \cap \frak{p}_-)=-\{ \beta \in \Delta_n^+ : \beta \le \phi_1+\cdots+\phi_{l-2}+\phi_a \}$ for $\lambda= 2\omega_b-\omega_1; 
\Delta(\frak{u}_\lambda \cap \frak{p}_+)=\{ \beta \in \Delta_n^+ : \beta \ge \phi_1+\cdots+\phi_{l-2}+\phi_{l-1}+\phi_l\}, 
\Delta(\frak{u}_\lambda \cap \frak{p}_-)=-\{ \beta \in \Delta_n^+ : \beta \le \phi_1 +\cdots + \phi_{l-2}+\phi_a \}$ for $\lambda=\omega_a + 2\omega_b-2\omega_1;
\Delta(\frak{u}_\lambda \cap \frak{p}_+)=\{ \beta \in \Delta_n^+ : \beta \ge \phi_1+\cdots+\phi_{j-1}+2\phi_j+\cdots+2\phi_{l-2}+\phi_{l-1}+\phi_l\}(2\le j \le l-2), 
\Delta(\frak{u}_\lambda \cap \frak{p}_-)=-\{ \beta \in \Delta_n^+ : \beta \le \phi_1 +\cdots + \phi_{l-2}+\phi_a \}$ for $\lambda= \omega_j+\omega_b-2\omega_1 (a=l-1,l).$ Also 
$\Delta(\frak{u}_\lambda \cap \frak{p}_+)=\phi, \Delta(\frak{u}_\lambda \cap \frak{p}_-)=-\{ \beta \in \Delta_n^+ : \beta \le \phi_1 +\cdots + \phi_{l-2}+\phi_a \}$ for 
$\lambda= \omega_b-\omega_1.$ Thus the number of equivalence classes of irreducible unitary representations with non-zero $(\frak{g}, K)$-cohomology for which 
$\Delta(\frak{u}_\lambda \cap \frak{p}_-) = -\{ \beta \in \Delta_n^+ : \beta \le \phi_1 +\cdots + \phi_{l-2}+\phi_a \}(a=l-1,l),$ is $l.$ \\    
Let $\Delta(\frak{u}_\lambda \cap \frak{p}_-)=-\{ \beta \in \Delta_n^+ : \beta \le \xi_1 \textrm{ or } \xi_2 \}.$ Then 
$\Delta(\frak{u}_\lambda \cap \frak{p}_+)=\{ \beta \in \Delta_n^+ : \beta \ge \phi_1+\cdots+\phi_{l-2}+\phi_{l-1}+\phi_l\}, 
\Delta(\frak{u}_\lambda \cap \frak{p}_-)=-\{ \beta \in \Delta_n^+ : \beta \le \xi_1 \textrm{ or } \xi_2 \}$ for $\lambda= 2\omega_{l-1}+2\omega_l-3\omega_1; $
$\Delta(\frak{u}_\lambda \cap \frak{p}_+)=\{ \beta \in \Delta_n^+ : \beta \ge \phi_1+\cdots+\phi_{j-1}+2\phi_j+\cdots+2\phi_{l-2}+\phi_{l-1}+\phi_l\}(2\le j \le l-2), 
\Delta(\frak{u}_\lambda \cap \frak{p}_-)=-\{ \beta \in \Delta_n^+ : \beta \le \xi_1 \textrm{ or } \xi_2 \}$ for $\lambda= \omega_j+\omega_{l-1}+\omega_l-3\omega_1;$ and 
$\Delta(\frak{u}_\lambda \cap \frak{p}_+)=\phi, 
\Delta(\frak{u}_\lambda \cap \frak{p}_-)=-\{ \beta \in \Delta_n^+ : \beta \le \xi_1 \textrm{ or } \xi_2 \}$ for $\lambda= \omega_{l-1}+\omega_l-2\omega_1.$ 
 Thus the number of equivalence classes of irreducible unitary representations with non-zero $(\frak{g}, K)$-cohomology for which 
$\Delta(\frak{u}_\lambda \cap \frak{p}_-) = -\{\beta \in \Delta_n^+ : \beta \le \xi_1 \textrm{ or } \xi_2 \}$ is $l-1.$ \\ 
Let $\Delta(\frak{u}_\lambda \cap \frak{p}_-)=-\{ \beta \in \Delta_n^+ : \beta \le \phi_1 +\cdots + \phi_{l-2}+\phi_{l-1}+\phi_l \}.$ Then 
$\Delta(\frak{u}_\lambda \cap \frak{p}_+)=\{ \beta \in \Delta_n^+ : \beta \ge \phi_1+\cdots+\phi_{j-1}+2\phi_j+\cdots+2\phi_{l-2}+\phi_{l-1}+\phi_l\}(2\le j \le l-2), 
\Delta(\frak{u}_\lambda \cap \frak{p}_-)=-\{ \beta \in \Delta_n^+ : \beta \le \phi_1 +\cdots + \phi_{l-2}+\phi_{l-1}+\phi_l \}$ for $\lambda= \omega_j+\omega_{l-2}-3\omega_1; 
\Delta(\frak{u}_\lambda \cap \frak{p}_+)=\phi,
\Delta(\frak{u}_\lambda \cap \frak{p}_-)=-\{ \beta \in \Delta_n^+ : \beta \le \phi_1 +\cdots + \phi_{l-2}+\phi_{l-1}+\phi_l \}$ for $\lambda= \omega_{l-2}-2\omega_1.$ Thus 
the number of equivalence classes of irreducible unitary representations with non-zero $(\frak{g}, K)$-cohomology for which 
$\Delta(\frak{u}_\lambda \cap \frak{p}_-) = -\{\beta \in \Delta_n^+ : \beta \le \phi_1 +\cdots + \phi_{l-2}+\phi_{l-1}+\phi_l \},$ is $l-2.$ \\ 
Let $\Delta(\frak{u}_\lambda \cap \frak{p}_-)=-\{ \beta \in \Delta_n^+ : \beta \le \phi_1+\cdots+\phi_{i-1}+2\phi_i+\cdots+2\phi_{l-2}+\phi_{l-1}+\phi_l\}(2\le i \le l-2), l \ge 4.$ Then 
$\Delta(\frak{u}_\lambda \cap \frak{p}_+)=\phi, 
\Delta(\frak{u}_\lambda \cap \frak{p}_-)=-\{ \beta \in \Delta_n^+ : \beta \le \phi_1+\cdots+\phi_{i-1}+2\phi_i+\cdots+2\phi_{l-2}+\phi_{l-1}+\phi_l\}$ for 
$\lambda= \omega_{i-1}-2\omega_1$ if $3\le i \le l-2,$ and $\lambda = -\omega_1$ if $i=2.$ Also 
$\Delta(\frak{u}_\lambda \cap \frak{p}_+)=\{ \beta \in \Delta_n^+ : \beta \ge \phi_1+\cdots+\phi_{j-1}+2\phi_j+\cdots+2\phi_{l-2}+\phi_{l-1}+\phi_l\}(2\le j \le i-1), 
\Delta(\frak{u}_\lambda \cap \frak{p}_-)=-\{ \beta \in \Delta_n^+ : \beta \le \phi_1+\cdots+\phi_{i-1}+2\phi_i+\cdots+2\phi_{l-2}+\phi_{l-1}+\phi_l\}$ for 
$\lambda= \omega_{i-1}+\omega_j-3\omega_1.$ Thus 
the number of equivalence classes of irreducible unitary representations with non-zero $(\frak{g}, K)$-cohomology for which 
$\Delta(\frak{u}_\lambda \cap \frak{p}_-) = -\{\beta \in \Delta_n^+ : \beta \le \phi_1+\cdots+\phi_{i-1}+2\phi_i+\cdots+2\phi_{l-2}+\phi_{l-1}+\phi_l\}(2\le i \le l-2),$ is $i-1.$ \\ 
Hence $A=(l+2)+(l+1)+\cdots+5+4+l+l+(l-1)+(l-2)+(l-3)+\cdots+1=l^2+4l-3.$

{\bf $\frak{g} = \frak{\delta}_2:$} In this case $\Delta_n^+ =\{\phi_1,\phi_2\}.$ If $\lambda = 0,$ then $\Delta(\frak{u}_\lambda \cap \frak{p}_-), 
\Delta(\frak{u}_\lambda \cap \frak{p}_+)$ are empty. 
$\lambda = \omega_i(1\le i \le2)$ implies $\Delta(\frak{u}_\lambda \cap \frak{p}_-)=\phi, \Delta(\frak{u}_\lambda \cap \frak{p}_+)=\{\phi_i\}; 
\lambda = \omega_1+\omega_2$ implies $\Delta(\frak{u}_\lambda \cap \frak{p}_-)=\phi, \Delta(\frak{u}_\lambda \cap \frak{p}_+)=\{\phi_1,\phi_2\}; 
\lambda = -\omega_i(1\le i \le2)$ implies $\Delta(\frak{u}_\lambda \cap \frak{p}_-)=\{-\phi_i\}, \Delta(\frak{u}_\lambda \cap \frak{p}_+)=\phi; 
\lambda = \omega_1-\omega_2$ implies $\Delta(\frak{u}_\lambda \cap \frak{p}_-)=\{-\phi_2\}, \Delta(\frak{u}_\lambda \cap \frak{p}_+)=\{\phi_1\}; 
\lambda = \omega_2-\omega_1$ implies $\Delta(\frak{u}_\lambda \cap \frak{p}_-)=\{-\phi_1\}, \Delta(\frak{u}_\lambda \cap \frak{p}_+)=\{\phi_2\};$ and 
$\lambda = -\omega_1-\omega_2$ implies $\Delta(\frak{u}_\lambda \cap \frak{p}_+)=\phi, \Delta(\frak{u}_\lambda \cap \frak{p}_-)=\{-\phi_1,-\phi_2\}.$ Thus $A=9=l^2+4l-3.$  

(ii) An irreducible unitary representation $\pi$ of $SO_0(2,m)$ with trivial infinitesimal character is a discrete series representation {\it if and only if} $\pi$ is unitarily equivalent to 
$A_\frak{b},$ where $\frak{b}$ is a Borel subalgebra of $\frak{g}$ containing $\frak{h} + \sum_{\alpha \in \Delta_\frak{k}^+} \frak{g}^\alpha.$ 
If $\frak{b}$ is a Borel subalgebra of $\frak{g}$ containing $\frak{h} + \sum_{\alpha \in \Delta_\frak{k}^+} \frak{g}^\alpha,$ then 
$\frak{b}=\frak{b}_\lambda=\frak{h} \oplus \frak{u}_\lambda$ 
for some linear function $\lambda$ on $\frak{h}_\mathbb{R}$ with $\langle \lambda,\alpha \rangle \neq 0$ for all $\alpha \in \Delta,$ and 
$\langle \lambda,\alpha \rangle >  0$ for all $\alpha \in \Delta_\frak{k}^+.$ Since $\langle \lambda,\beta \rangle \neq 0$ for all $\beta \in \Delta_n^+,$ 
for the irreducible unitary representation $A_\frak{b}$ we have 
$(-\Delta(\frak{u}_\lambda \cap \frak{p}_-)) \cup \Delta(\frak{u}_\lambda \cap \frak{p}_+) = \Delta_n^+.$ 

  Conversely suppose that $\lambda$ be a linear function on $\frak{h}_\mathbb{R}$ such that $\langle \lambda,\alpha \rangle \ge  0$ for all $\alpha \in \Delta_\frak{k}^+,$  and 
$(-\Delta(\frak{u}_\lambda \cap \frak{p}_-)) \cup \Delta(\frak{u}_\lambda \cap \frak{p}_+) = \Delta_n^+.$ Since $\langle \lambda,\alpha \rangle \ge  0$ for all 
$\alpha \in \Delta_\frak{k}^+,$ and $\langle \lambda,\beta \rangle \neq 0$ for all $\beta \in \Delta_n, 
\lambda=\sum_{1\le i\le l} \frac{c_i\omega_i}{\langle \omega_i,\phi_i \rangle},$ where $c_1$ is a non-zero real number and $c_i$ is a non-negative real number for all $2 \le i \le l.$ 
If $c_1 >0,$ let $d_1=c_1;$ and $d_i =c_i \textrm{ if } c_i \neq 0,\textrm{and }d_i =1 \textrm{ if } c_i = 0;$ for all $2\le i \le l.$ 
If $c_1 <0,$ let $\{i: 2\le i \le l, c_i = 0\}=\{i_1,i_2, \ldots , i_k\},$ and 
$m_j =\textrm{max} \{ n_{\phi_{i_j}}(\beta): \beta \in \Delta_n^+, \langle \lambda,\beta \rangle < 0\}$ for all $1\le j \le k.$ 
Assume that $\frak{g} \neq \delta_2,$ and if $\frak{g}=\delta_l (l \ge 3),$ 
$-\Delta(\frak{u}_\lambda \cap \frak{p}_-) \neq \{\beta \in \Delta_n^+ : \beta \le \xi_i\},$ where $i=1,2.$ Then if $\beta \in -\Delta(\frak{u}_\lambda \cap \frak{p}_-), 
\beta' \in \Delta(\frak{u}_\lambda \cap \frak{p}_+); \beta < \beta',$ and so $n_{\phi_{i_j}}(\beta) \le n_{\phi_{i_j}}(\beta')$ for all $1 \le j \le k.$ 
In this case, if $c_1<0,$ let $d_1=c_1-\sum_{1\le j \le k}m_j;$ and $d_i =c_i \textrm{ if } c_i \neq 0,\textrm{and }d_i =1 \textrm{ if } c_i = 0;$ for all $2\le i \le l.$ Let 
$\lambda'=\sum_{1\le i\le l} \frac{d_i\omega_i}{\langle \omega_i,\phi_i \rangle},$ for any $c_1 \in \mathbb{R}\setminus \{0\}.$ 
Then $\langle \lambda',\alpha \rangle >  0$ for all $\alpha \in \Delta_\frak{k}^+;$ and 
if $\beta \in \Delta_n,\ \langle \lambda,\beta \rangle < 0 \implies \langle \lambda',\beta \rangle < 0,$ and $\langle \lambda,\beta \rangle > 0 \implies 
\langle \lambda',\beta \rangle > 0.$ Thus $\frak{q}_{\lambda'}$ is a Borel subalgebra of $\frak{g}$ containing $\frak{h} + \sum_{\alpha \in \Delta_\frak{k}^+} \frak{g}^\alpha,$ and 
$\Delta(\frak{u}_{\lambda'} \cap \frak{p}_-) = \Delta(\frak{u}_\lambda \cap \frak{p}_-), \Delta(\frak{u}_{\lambda'} \cap \frak{p}_+)=\Delta(\frak{u}_\lambda \cap \frak{p}_+).$ 
Thus $A_{\frak{q}_\lambda}$ is is unitarily equivalent to $A_{\frak{q}_{\lambda'}},$ which is a discrete series representation with trivial infinitesimal character. \\ 
Let $\frak{g}=\delta_l (l \ge 3),$ and $-\Delta(\frak{u}_\lambda \cap \frak{p}_-) = \{\beta \in \Delta_n^+ : \beta \le \xi_i\},$ where $i=1,2;$ that is 
$-\Delta(\frak{u}_\lambda \cap \frak{p}_-) = \{\beta \in \Delta_n^+ : \beta \le \phi_1 +\cdots + \phi_{l-2}+\phi_a \}(a=l-1,l).$ Then 
$\Delta(\frak{u}_\lambda \cap \frak{p}_+) = \{\beta \in \Delta_n^+ : \beta \ge \phi_1 +\cdots + \phi_{l-2}+\phi_b \},$ where $b=l \textrm{ if }a=l-1, \textrm{and } b=l-1 \textrm{ if }a=l.$ 
Clearly $\langle \lambda , \phi_b \rangle > 0,$ that is $c_b>0.$ If $c_a=0,$ let $d_1=c_1-\sum_{1\le j \le k}m_j; d_a =c_a+1; d_b=c_b+1;$ 
and for all $2\le i \le l-2,\ d_i =c_i \textrm{ if } c_i \neq 0,\textrm{and }d_i =1 \textrm{ if } c_i = 0.$ If $c_a \neq 0,$ 
let $d_1=c_1-\sum_{1\le j \le k}m_j;$ and $d_i =c_i \textrm{ if } c_i \neq 0,\textrm{and }d_i =1 \textrm{ if } c_i = 0;$ for all $2\le i \le l.$ 
$\lambda'=\sum_{1\le i\le l} \frac{d_i\omega_i}{\langle \omega_i,\phi_i \rangle}.$ Since 
$\beta \in -\Delta(\frak{u}_\lambda \cap \frak{p}_-), \beta' \in \Delta(\frak{u}_\lambda \cap \frak{p}_+) \implies n_{\phi_{i_j}}(\beta) \le n_{\phi_{i_j}}(\beta')$ for all $1 \le j \le k,i_j \neq l-1,l;$ 
we have $\langle \lambda',\alpha \rangle >  0$ for all $\alpha \in \Delta_\frak{k}^+;$ and 
if $\beta \in \Delta_n,\ \langle \lambda,\beta \rangle < 0 \implies \langle \lambda',\beta \rangle < 0,$ and $\langle \lambda,\beta \rangle > 0 \implies \langle \lambda',\beta \rangle > 0.$ 
As above $A_{\frak{q}_\lambda}$ is is unitarily equivalent to $A_{\frak{q}_{\lambda'}},$ which is a discrete series representation with trivial infinitesimal character. \\ 
Let $\frak{g}=\frak{\delta}_2.$ Since $(-\Delta(\frak{u}_\lambda \cap \frak{p}_-)) \cup \Delta(\frak{u}_\lambda \cap \frak{p}_+) = \Delta_n^+,$ the candidates of 
$(-\Delta(\frak{u}_\lambda \cap \frak{p}_-), \Delta(\frak{u}_\lambda \cap \frak{p}_+))$ are $(\phi, \{\phi_1,\phi_2\}), (\{\phi_1\},\{\phi_2\}),(\{\phi_2\},\{\phi_1\}),$ and 
$(\{\phi_1,\phi_2\},\phi).$ The corresponding $\lambda'$ are $\omega_1+\omega_2,-\omega_1+\omega_2,\omega_1-\omega_2,-\omega_1-\omega_2$ resectively. Then 
$A_{\frak{q}_\lambda}$ is is unitarily equivalent to $A_{\frak{q}_{\lambda'}},$ which is a discrete series representation with trivial infinitesimal character. 

  The Blattner parameter of the discrete series representation $A_{\frak{b}_\lambda},$ where $\frak{b}_\lambda = \frak{h} \oplus \frak{u}_\lambda$ 
for some linear function $\lambda$ on $\frak{h}_\mathbb{R}$ with $\langle \lambda,\alpha \rangle \neq 0$ for all $\alpha \in \Delta,$ and 
$\langle \lambda,\alpha \rangle >  0$ for all $\alpha \in \Delta_\frak{k}^+,$ is $\sum_{\beta \in \Delta(\frak{u}_\lambda \cap \frak{p}_-) \cup \Delta(\frak{u}_\lambda \cap \frak{p}_+)} \beta.$ 
If $m \neq 2, \frak{g}$ is simple and in this case the discrete series representation $A_{\frak{b}_\lambda}$ is a holomorphic discrete series representation {\it if and only if} the  Blattner parameter is 
$\sum_{\beta \in \Delta_n^+} \beta,$ or $\sum_{\beta \in \Delta_n^-} \beta;$ that is $\Delta(\frak{u}_\lambda \cap \frak{p}_+)$ is either $\Delta_n^+$ or empty. 
For since $\frak{g}$ is simple, the only Borel-de Siebenthal positive root system containing $\Delta_\frak{k}^+$ are $\Delta_\frak{k}^+ \cup \Delta_n^+,$ and $\Delta_\frak{k}^+ \cup \Delta_n^-.$
Hence the number of equivalence classes of holomorphic discrete series representations of $SO_0(2,m)(m \neq 2)$ with trivial infinitesimal character is $2.$ If $\frak{g}=\frak{\delta}_2,$ any 
positive root system is Borel-de Siebenthal positive root system, and so any discrete series representation with trivial infinitesimal character is holomorphic. 

  Let $\frak{g}=\frak{b}_l(l \ge 2).$ The candidates of $(\Delta(\frak{u}_\lambda \cap \frak{p}_-), \Delta(\frak{u}_\lambda \cap \frak{p}_+))$ for which 
$(-\Delta(\frak{u}_\lambda \cap \frak{p}_-)) \cup \Delta(\frak{u}_\lambda \cap \frak{p}_+) = \Delta_n^+,$ are 
$(\phi ,\{ \beta \in \Delta_n^+ : \beta \ge \phi_1\}), (-\{ \beta \in \Delta_n^+ : \beta \le \phi_1 +\cdots + \phi_i \}, \{ \beta \in \Delta_n^+ : \beta \ge \phi_1 +\cdots + \phi_{i+1}\})
(1\le i \le l-1), (-\{ \beta \in \Delta_n^+ : \beta \le \phi_1 +\cdots + \phi_l \}, \{ \beta \in \Delta_n^+ : \beta \ge \phi_1 +\cdots + \phi_{l-1}+2\phi_l \}), 
(-\{ \beta \in \Delta_n^+ : \beta \le \phi_1 +\cdots + \phi_{i-1}+2\phi_i + \cdots+2\phi_l \}, \{ \beta \in \Delta_n^+ : \beta \ge \phi_1 +\cdots + \phi_{i-2}+2\phi_{i-1} + \cdots+2\phi_l \})
(3 \le i \le l), (-\{ \beta \in \Delta_n^+ : \beta \le \phi_1+2\phi_2 + \cdots+2\phi_l \}, \phi ).$ Thus the number of equivalence classes of discrete series representations of $SO_0(2,m)$ 
with trivial infinitesimal character is $2l \textrm{ if }m=2l-1, l \ge 2.$   \\ 
If $\frak{g}=\frak{b}_1,$ the candidates of $(\Delta(\frak{u}_\lambda \cap \frak{p}_-), \Delta(\frak{u}_\lambda \cap \frak{p}_+))$ for which 
$(-\Delta(\frak{u}_\lambda \cap \frak{p}_-)) \cup \Delta(\frak{u}_\lambda \cap \frak{p}_+) = \Delta_n^+,$ are $(\phi ,\{\phi_1\}), (\{-\phi_1\},\phi)$. Thus the number is $2.$ \\
Let $\frak{g}=\frak{\delta}_l(l \ge 3).$ The candidates of $(\Delta(\frak{u}_\lambda \cap \frak{p}_-), \Delta(\frak{u}_\lambda \cap \frak{p}_+))$ for which 
$(-\Delta(\frak{u}_\lambda \cap \frak{p}_-)) \cup \Delta(\frak{u}_\lambda \cap \frak{p}_+) = \Delta_n^+,$ are 
$(\phi ,\{ \beta \in \Delta_n^+ : \beta \ge \phi_1\}), (-\{ \beta \in \Delta_n^+ : \beta \le \phi_1 +\cdots + \phi_i \}, \{ \beta \in \Delta_n^+ : \beta \ge \phi_1 +\cdots + \phi_{i+1}\})
(1\le i \le l-3,l\ge 4), (-\{ \beta \in \Delta_n^+ : \beta \le \phi_1 +\cdots + \phi_{l-2} \}, \{ \beta \in \Delta_n^+ : \beta \ge \xi_1 \textrm{ or } \xi_2\}), 
(-\{ \beta \in \Delta_n^+ : \beta \le \xi_1 \}, \{ \beta \in \Delta_n^+ : \beta \ge \xi_2\}), (-\{ \beta \in \Delta_n^+ : \beta \le \xi_2 \}, \{ \beta \in \Delta_n^+ : \beta \ge \xi_1\}), 
(-\{ \beta \in \Delta_n^+ : \beta \le \xi_1 \textrm{ or } \xi_2\}, \{ \beta \in \Delta_n^+ : \beta \ge \phi_1 +\cdots + \phi_l \}),  
(-\{ \beta \in \Delta_n^+ : \beta \le \phi_1 +\cdots + \phi_l \}, \{ \beta \in \Delta_n^+ : \beta \ge \phi_1 +\cdots +\phi_{l-3}+2\phi_{l-2}+ \phi_{l-1}+\phi_l \}), 
(-\{ \beta \in \Delta_n^+ : \beta \le \phi_1 +\cdots + \phi_{i-1}+2\phi_i + \cdots+2\phi_{l-2}+\phi_{l-1}+\phi_l \}, 
\{ \beta \in \Delta_n^+ : \beta \ge \phi_1 +\cdots + \phi_{i-2}+2\phi_{i-1} + \cdots+2\phi_{l-2}+\phi_{l-1}+\phi_l \})
(3 \le i \le l-2,l\ge 4), (-\{ \beta \in \Delta_n^+ : \beta \le \phi_1+2\phi_2 + \cdots+2\phi_{l-2}+\phi_{l-1}+\phi_l \}, \phi ).$ 
Thus the number of equivalence classes of discrete series representations of $SO_0(2,m)$ with trivial infinitesimal character is $2l \textrm{ if }m=2l-2, l \ge 3.$   \\  
If $\frak{g}=\frak{\delta}_2,$ then obviously the number is $4.$

\begin{remark} 
Note that the set of all Hodge types $(R_+(\frak{q}), R_-(\frak{q}))$ of irreducible unitary representations $A_\frak{q}$ is given by 
\[\
\begin{cases}
\{(i,j): i,j \in \mathbb{N}\cup \{0\}, l \le i+j \le |\Delta_n^+|\}\cup\{(i,i): 0 \le i \le [\frac{|\Delta_n^+|}{2}]\} & \textrm{if } \frak{g}=\frak{b}_l, \\
\{(i,j): i,j \in \mathbb{N}\cup \{0\}, l-1 \le i+j \le |\Delta_n^+|\}\cup\{(i,i): 0 \le i \le [\frac{|\Delta_n^+|}{2}]\} & \textrm{if } \frak{g}=\frak{\delta}_l; \\
\end{cases}
\] where $|\Delta_n^+|$ is the number of roots in $\Delta_n^+.$ 
The discrete series representations with trivial infinitesimal character correspond to the set $\{(i,j): i,j \in \mathbb{N}\cup \{0\}, i+j = |\Delta_n^+|\}.$  
\end{remark}

\noindent
\section{Poincar\'{e} polynomials of cohomologies of the irreducible unitary representations with non-zero $(\frak{g},K)$-cohomology of the Lie group $SO_0(2,m)$}

 To determine the Poincar\'{e} polynomial of $H^*(\frak{g},K;A_{\frak{q},K}),$ we need to determine the spaces $Y_\frak{q},$ and to do so 
we need to analyze the $\theta$-stable parabolic 
subalgebras 
containing $\frak{h} \oplus \sum_{\alpha \in \Delta_\frak{k}^+}\frak{g}^{\alpha}$ 
more closely. 
Note that a parabolic subalgebra of $\frak{g}$ is $\theta$-stable {\it if and only if} it contains a maximal abelian subspace of $\frak{k}_0.$ Thus a parabolic $\frak{q}$ 
subalgebra of $\frak{g}$ which contains $\frak{h} \oplus \sum_{\alpha \in \Delta_\frak{k}^+}\frak{g}^{\alpha},$ is $\theta$-stable. So there exists a positive root system 
$\Delta_\frak{q}^+$ of $\Delta$ containing $\Delta_\frak{k}^+$ and a subset $\Gamma$ of $\Phi_\frak{q},$ the set of all simple roots in $\Delta_\frak{q}^+,$ such that 
$\frak{q}= \frak{l} \oplus \frak{u},$
where $\frak{l} = \frak{h} \oplus \sum_{n_\phi (\alpha) =0 \textrm{ for all } \phi \in \Gamma} \frak{g}^\alpha$ is the Levi subalgebra of $\frak{q},$ and 
$\frak{u} = \sum_{n_\phi (\alpha) >0 \textrm{ for some } \phi \in \Gamma} \frak{g}^\alpha$ is the nilradical of $\frak{q};$ where  
$\alpha =\sum_{\phi \in \Phi_\frak{q}}  n_\phi (\alpha) \phi \in \Delta.$ Note that the Levi subalgebra $\frak{l}$ is the direct sum of an $|\Gamma|$-dimensional centre and 
a semisimple Lie algebra whose Dynkin diagram is the subdiagram of the dynkin diagram of $\frak{g}$ consisting of the vertices $\Phi_\frak{q} \setminus \Gamma.$ 

If $\frak{q}$ is a parabolic subalgebra which contains $\frak{h} \oplus \sum_{\alpha \in \Delta_\frak{k}^+}\frak{g}^{\alpha},$ there are many positive root systems of $\Delta$ containing 
$\Delta_\frak{k}^+ \cup \Delta(\frak{u} \cap \frak{p}_-) \cup  \Delta(\frak{u} \cap \frak{p}_+).$ For example, $\Delta_\frak{k}^+ \cup \Delta(\frak{u} \cap \frak{p}_-) \cup 
(\Delta_n^+\setminus (-\Delta(\frak{u} \cap \frak{p}_-)))$ is a positive root system of $\Delta,$ as we have seen in the proof of Th. \ref{th1}(ii) that there exists a non-singular linear 
function $\lambda'$ on $\frak{h}_\mathbb{R}$ such that $\lambda'$ is dominant with respect to $\Delta_\frak{k}^+ \cup \Delta(\frak{u} \cap \frak{p}_-) \cup 
(\Delta_n^+\setminus (-\Delta(\frak{u} \cap \frak{p}_-))).$ We define $\Delta_\frak{q}^+= \Delta_\frak{k}^+ \cup \Delta(\frak{u} \cap \frak{p}_-) \cup 
(\Delta_n^+\setminus (-\Delta(\frak{u} \cap \frak{p}_-))).$ In the Tables \ref{b-table} and \ref{d-table}, we have determined $\Phi_\frak{q}, \Gamma, Y_\frak{q}$ for each 
$\theta$-stable parabolic subalgebra containing $\frak{h} \oplus \sum_{\alpha \in \Delta_\frak{k}^+}\frak{g}^{\alpha}.$ 

In the Tables \ref{b-table} and \ref{d-table}, we can see that $Y_\frak{q}$ is either singleton, or $\frac{SU(k)}{S(U(1)\times U(k-1))}(k \ge 2),$ or $\frac{SO(2k+1)}{SO(2)\times SO(2k-1)}
(k \ge 1),$ or $\frac{SO(2k)}{SO(2)\times SO(2k-2)}(k \ge 2).$ We have 
$P(\textrm{singleton},t)=1,$ $P(\frac{SU(k)}{S(U(1)\times U(k-1))},t)=1+t^2+t^4+\cdots+t^{2k-2} \textrm{ for all }k \ge 2, 
P(\frac{SO(2k+1)}{SO(2)\times SO(2k-1)},t)=1+t^2+t^4+\cdots+t^{4k-2} \textrm{ for all }k \ge 1, 
P(\frac{SO(2k)}{SO(2)\times SO(2k-2)},t)=1+t^2+t^4+\cdots+t^{2k-4}+2t^{2k-2}+t^{2k}+\cdots+t^{4k-4} \textrm{ for all }k \ge 2.$ See \cite{ghv}. Since 
$H^r (\frak{g}, K; A_{\frak{q}, K}) =  H^{p,q} (\frak{g}, K; A_{\frak{q}, K})
\cong H^{p-R_+(\frak{q}),q-R_-(\frak{q})} (Y_\frak{q} ;\mathbb{C}),$ for unique non-negative integers $p,q$ with 
$p+q=r, p-q=R_+(\frak{q})-R_-(\frak{q});$ 
we write two variable Poincar\'{e} polynomial $P_\frak{q}(x,t)$ for $H^*(\frak{g},K;A_{\frak{q},K}),$ and the coefficient of the term $x^pt^q$ in $P_\frak{q}(x,t)$ is 
dim$( H^{p,q} (\frak{g}, K; A_{\frak{q}, K})).$

\begin{landscape} 
\begin{table} 
\caption{Poincar\'{e} Polynomials of The Irreducible Unitary Representations $A_\frak{q}$ of $SO_0(2,2l-1)(l \ge 1):$}\label{b-table}
\begin{tabular}{||c|c|c|c|c|c||}
\hline
\begin{tabular}{c} $\Delta(\frak{u} \cap \frak{p}_-),$\\$R_-(\frak{q})$ \end{tabular} & $\Phi_\frak{q}$ & 
\begin{tabular}{c} $\Delta(\frak{u} \cap \frak{p}_+),$\\$R_+(\frak{q})$ \end{tabular} & $\Gamma$ & $Y_\frak{q}$ & $P_\frak{q}(x,t)$ \\
\hline
\hline
\begin{tabular}{c} empty, \\$0$ \end{tabular} & 
\begin{tabular}{c} $\{\nu, \phi_2, \ldots ,\phi_l\};$\\$\nu=\phi_1.$\\ See (a).  \end{tabular} & 
\begin{tabular}{c} $\{\beta \in \Delta_n^+ :\beta \ge $\\$ \nu\},2l-1$\\ \\ 
$\{\beta \in \Delta_n^+ : \beta \ge $\\$\phi_1+\cdots $\\$+\phi_i\},2l-i$\\$(2 \le i \le l)$\\ \\empty, $0$ \end{tabular} & 
\begin{tabular}{c} $\{\nu\}$\\ \\$\{\phi_i\}$\\$(2 \le i \le l)$\\ \\ empty \end{tabular} & 
\begin{tabular}{c} singleton \\ \\$\frac{SU(i)}{S(U(1)\times U(i-1))}$\\$(2 \le i \le l)$\\  \\$\frac{SO(2l+1)}{SO(2)\times SO(2l-1)}$ \end{tabular} & 
\begin{tabular}{c} $x^{2l-1}$\\ \\$x^{2l-i}+x^{2l-i+1}t$\\$+\cdots+x^{2l-1}t^{i-1}$\\$(2 \le i \le l)$\\ \\$1+xt+\cdots$\\$+x^{2l-1}t^{2l-1}$   \end{tabular} \\ 
\hline 
\begin{tabular}{c} $-\{\beta \in \Delta_n^+ : \beta \le $\\$\phi_1+\cdots +\phi_i\},$\\$i$\\$(1 \le i \le l-1)$ \end{tabular} & 
\begin{tabular}{c} $\{\nu_1,\nu_2,\phi_2,\ldots ,$\\$\phi_i,\phi_{i+2},\ldots , \phi_l\};\nu_1$\\$=-(\phi_1+\cdots +\phi_i),$\\$\nu_2=\phi_1+\cdots+\phi_{i+1}.$\\ See (b). \end{tabular} & 
\begin{tabular}{c} $\{\beta \in \Delta_n^+ : \beta \ge $\\$\nu_2\},2l-1-i$\\ \\$\{\beta \in \Delta_n^+ : \beta \ge $\\$\phi_1+\cdots $\\$+\phi_j\},2l-j$\\$(i+2 \le j \le l)$\\ \\
$\{\beta \in \Delta_n^+ : \beta \ge $\\$\phi_1+\cdots +\phi_i+$\\$2\phi_{i+1}+\cdots $\\$+2\phi_l\},i$ \end{tabular} & 
\begin{tabular}{c} $\{\nu_1,\nu_2\}$\\ \\$\{\nu_1,\phi_j\}$\\$(i+2\le j \le l)$\\ \\$\{\nu_1\}$ \end{tabular} & 
\begin{tabular}{c} singleton \\ \\ $\frac{SU(j-i)}{S(U(1)\times U(j-i-1))}$\\$(i+2 \le j \le l)$\\ \\$\frac{SO(2l-2i+1)}{SO(2)\times SO(2l-2i-1)}$ \end{tabular} & 
\begin{tabular}{c} $x^{2l-1-i}t^i$\\ \\$x^{2l-j}t^i$\\$+x^{2l-j+1}t^{i+1}+$\\$\cdots+x^{2l-i-1}t^{j-1}$\\$(i+2 \le j \le l)$\\ \\$x^it^i$\\$+x^{i+1}t^{i+1}+\cdots$\\$+x^{2l-i-1}t^{2l-i-1}$   \end{tabular} \\  
\hline 
\begin{tabular}{c} $-\{\beta \in \Delta_n^+ : \beta \le $\\$\phi_1+\cdots +\phi_l\},$\\$l$ \end{tabular} & 
\begin{tabular}{c} $\{\nu_1,\nu_2,\phi_2,$\\$\ldots ,\phi_{l-1}\};\nu_1$\\$= -(\phi_1+\cdots +\phi_l),$\\$\nu_2=\phi_1+\cdots$\\$+\phi_{l-1}+2\phi_l.$\\ See (c). \end{tabular} & 
\begin{tabular}{c} $\{\beta \in \Delta_n^+ :\beta \ge $\\$\nu_2\},l-1$\\ \\
$\{\beta \in \Delta_n^+ : \beta \ge $\\$\phi_1+\cdots +\phi_{j-1}$\\$+2\phi_j+\cdots$\\$+2\phi_l\},j-1$\\$(2 \le j \le l-1)$\\ \\ empty, $0$ \end{tabular} & 
\begin{tabular}{c} $\{\nu_1,\nu_2\}$\\ \\$\{\nu_1,\phi_j\}$\\$(2\le j \le l-1)$\\ \\$\{\nu_1\}$  \end{tabular} & 
\begin{tabular}{c} singleton \\ \\ $\frac{SU(l-j+1)}{S(U(1)\times U(l-j))}$\\$(2 \le j \le l-1)$\\ \\$\frac{SU(l)}{S(U(1)\times U(l-1))}$  \end{tabular} & 
\begin{tabular}{c} $x^{l-1}t^l$\\ \\$x^{j-1}t^l$\\$+x^{j}t^{l+1}+$\\$\cdots+x^{l-1}t^{2l-j}$\\$(2 \le j \le l-1)$\\ \\$t^l$\\$+xt^{l+1}+\cdots$\\$+x^{l-1}t^{2l-1}$   \end{tabular} \\  
\hline
\end{tabular} 
\end{table} 
\end{landscape}

\begin{landscape} 
\begin{table} 
\begin{tabular}{||c|c|c|c|c|c||}
\hline
\begin{tabular}{c} $\Delta(\frak{u} \cap \frak{p}_-),$\\$R_-(\frak{q})$ \end{tabular} & $\Phi_\frak{q}$ & 
\begin{tabular}{c} $\Delta(\frak{u} \cap \frak{p}_+),$\\$R_+(\frak{q})$ \end{tabular} & $\Gamma$ & $Y_\frak{q}$ & $P_\frak{q}(x,t)$ \\
\hline
\hline
\begin{tabular}{c} $-\{\beta \in \Delta_n^+ : \beta \le $\\$\phi_1+\cdots +\phi_{i-1}$\\$+2\phi_i+\cdots+2\phi_l\},$\\$2l-i+1$\\$(3\le i \le l)$ \end{tabular} & 
\begin{tabular}{c} $\{\nu_1,\nu_2,\phi_2, \ldots,$\\$\phi_{i-2},\phi_i,\ldots ,\phi_l\};$\\
$\nu_1=-(\phi_1+\cdots +\phi_{i-1}$\\$+2\phi_i+\cdots+2\phi_l),\nu_2$\\$= \phi_1+\cdots+\phi_{i-2}$\\$+2\phi_{i-1}+\cdots+2\phi_l.$\\ See (d). \end{tabular} & 
\begin{tabular}{c} $\{\beta \in \Delta_n^+ : \beta \ge $\\$\nu_2\},i-2$\\ \\
$\{\beta \in \Delta_n^+ : \beta \ge $\\$\phi_1+\cdots +$\\$\phi_{j-1}+2\phi_j+\cdots$\\$+2\phi_l\},j-1$\\$(2 \le j \le i-2)$\\ \\ empty, $0$ \end{tabular} & 
\begin{tabular}{c} $\{\nu_1,\nu_2\}$\\ \\$\{\nu_1,\phi_j\}$\\$(2\le j \le i-2)$\\ \\$\{\nu_1\}$ \end{tabular} & 
\begin{tabular}{c} singleton \\ \\$\frac{SU(i-j)}{S(U(1)\times U(i-j-1))}$\\$(2 \le j \le i-2)$\\ \\$\frac{SU(i-1)}{S(U(1)\times U(i-2))}$  \end{tabular} & 
\begin{tabular}{c} $x^{i-2}t^{2l-i+1}$\\ \\$x^{j-1}t^{2l-i+1}$\\$+x^jt^{2l-i+2}+$\\$\cdots+x^{i-2}t^{2l-j}$\\$(2 \le j \le i-2)$\\ \\$t^{2l-i+1}$\\$+xt^{2l-i+2}+\cdots$\\$+x^{i-2}t^{2l-1}$   \end{tabular} \\  
\hline
\begin{tabular}{c} $-\{\beta \in \Delta_n^+ : \beta \le $\\$\phi_1+2\phi_2$\\$+\cdots+2\phi_l\},$\\$2l-1$ \end{tabular} & 
\begin{tabular}{c} $\{\nu, \phi_2,\ldots ,$\\$\phi_l\}; \nu=-(\phi_1+$\\$2\phi_2+\cdots+2\phi_l).$\\ See (e). \end{tabular} & 
\begin{tabular}{c} empty, $0$ \end{tabular} & 
$\{\nu \}$  & singleton & $t^{2l-1}$ \\ 
\hline
\end{tabular} 
\end{table}

\begin{figure}
\begin{center} 
\begin{tikzpicture}

\filldraw [black] (-15,0) circle [radius = 0.1]; 
\draw (-14,0) circle [radius = 0.1]; 
\draw (-12.5,0) circle [radius = 0.1]; 
\draw (-11,0) circle [radius = 0.1]; 
\draw (-10,0) circle [radius = 0.1]; 
\node [above] at (-15.05,0.05) {$\nu$}; 
\node [above] at (-14.05,0.05) {$\phi_2$}; 
\node [above] at (-12.45,0.05) {$\phi_i$}; 
\node [above] at (-10.75,0.05) {$\phi_{l-1}$}; 
\node [above] at (-9.95,0.05) {$\phi_l$}; 
\node [left] at (-15.5,0) {(a) $\frak{b}_l :$}; 
\draw (-14.9,0) -- (-14.1,0); 
\draw (-13.9,0) -- (-13.5,0);
\draw [dotted] (-13.5,0) -- (-13,0); 
\draw (-13,0) -- (-12.6,0); 
\draw (-12.4,0) -- (-12,0);
\draw [dotted] (-12,0) -- (-11.5,0); 
\draw (-11.5,0) -- (-11.1,0);
\draw (-10.1,0) -- (-10.2,0.1); 
\draw (-10.1,0) -- (-10.2,-0.1); 
\draw (-10.9,0.025) -- (-10.15,0.025); 
\draw (-10.9,-0.025) -- (-10.15,-0.025);

\filldraw [black] (-15,-1.5) circle [radius = 0.1]; 
\filldraw [black] (-14,-1.5) circle [radius = 0.1]; 
\draw (-13,-1.5) circle [radius = 0.1]; 
\draw (-11.5,-1.5) circle [radius = 0.1]; 
\draw (-10.5,-1.5) circle [radius = 0.1]; 
\node [above] at (-14.95,-1.45) {$\nu_1$}; 
\node [above] at (-13.95,-1.45) {$\nu_2$}; 
\node [above] at (-12.95,-1.45) {$\phi_3$}; 
\node [above] at (-11.25,-1.45) {$\phi_{l-1}$}; 
\node [above] at (-10.45,-1.45) {$\phi_l$}; 
\node [left] at (-15.5,-1.5) {(b) $\frak{b}_l :$}; 
\node [below] at (-12.5,-1.6) {$(i=1)$}; 
\draw (-14.9,-1.5) -- (-14.1,-1.5); 
\draw (-13.9,-1.5) -- (-13.1,-1.5); 
\draw (-12.9,-1.5) -- (-12.5,-1.5); 
\draw [dotted] (-12.5,-1.5) -- (-12,-1.5);
\draw (-12,-1.5) -- (-11.6,-1.5); 
\draw (-10.6,-1.5) -- (-10.7,-1.4); 
\draw (-10.6,-1.5) -- (-10.7,-1.6); 
\draw (-11.4,-1.475) -- (-10.65,-1.475); 
\draw (-11.4,-1.525) -- (-10.65,-1.525);

\draw (-8,-1.5) circle [radius = 0.1]; 
\draw (-6.5,-1.5) circle [radius = 0.1]; 
\filldraw [black] (-5.5,-1.5) circle [radius = 0.1]; 
\filldraw [black] (-4.5,-1.5) circle [radius = 0.1]; 
\draw (-3.5,-1.5) circle [radius = 0.1]; 
\draw (-2,-1.5) circle [radius = 0.1]; 
\draw (-1,-1.5) circle [radius = 0.1]; 
\node [above] at (-7.95,-1.45) {$\phi_2$}; 
\node [above] at (-6.45,-1.45) {$\phi_i$}; 
\node [above] at (-5.45,-1.45) {$\nu_1$}; 
\node [above] at (-4.45,-1.45) {$\nu_2$}; 
\node [above] at (-3.25,-1.45) {$\phi_{i+2}$}; 
\node [above] at (-1.75,-1.45) {$\phi_{l-1}$}; 
\node [above] at (-0.95,-1.45) {$\phi_l$}; 
\node [below] at (-4.5,-1.6) {$(2\le i \le l-2)$}; 
\draw (-7.9,-1.5) -- (-7.5,-1.5); 
\draw [dotted] (-7.5,-1.5) -- (-7,-1.5); 
\draw (-7,-1.5) -- (-6.6,-1.5); 
\draw (-6.4,-1.5) -- (-5.6,-1.5);
\draw (-5.4,-1.5) -- (-4.6,-1.5);
\draw (-4.4,-1.5) -- (-3.6,-1.5);
\draw (-3.4,-1.5) -- (-3,-1.5);
\draw [dotted] (-3,-1.5) -- (-2.5,-1.5);
\draw (-2.5,-1.5) -- (-2.1,-1.5);
\draw (-1.1,-1.5) -- (-1.2,-1.4); 
\draw (-1.1,-1.5) -- (-1.2,-1.6); 
\draw (-1.9,-1.475) -- (-1.15,-1.475); 
\draw (-1.9,-1.525) -- (-1.15,-1.525);

\draw (1.5,-1.5) circle [radius = 0.1]; 
\draw (3,-1.5) circle [radius = 0.1]; 
\filldraw [black] (4,-1.5) circle [radius = 0.1]; 
\filldraw [black] (5,-1.5) circle [radius = 0.1]; 
\node [above] at (1.55,-1.45) {$\phi_2$}; 
\node [above] at (3.25,-1.45) {$\phi_{l-1}$}; 
\node [above] at (4.05,-1.45) {$\nu_1$}; 
\node [above] at (5.05,-1.45) {$\nu_2$}; 
\node [below] at (3,-1.6) {$(i=l-1)$}; 
\draw (1.6,-1.5) -- (2,-1.5); 
\draw [dotted] (2,-1.5) -- (2.5,-1.5); 
\draw (2.5,-1.5) -- (2.9,-1.5); 
\draw (3.1,-1.5) -- (3.9,-1.5);
\draw (4.9,-1.5) -- (4.8,-1.4); 
\draw (4.9,-1.5) -- (4.8,-1.6); 
\draw (4.1,-1.475) -- (4.85,-1.475); 
\draw (4.1,-1.525) -- (4.85,-1.525);

\draw (-15,-3.5) circle [radius = 0.1]; 
\draw (-13.5,-3.5) circle [radius = 0.1]; 
\filldraw [black] (-12.5,-3.5) circle [radius = 0.1]; 
\filldraw [black] (-11.5,-3.5) circle [radius = 0.1]; 
\node [above] at (-14.95,-3.45) {$\phi_2$}; 
\node [above] at (-13.25,-3.45) {$\phi_{l-1}$}; 
\node [above] at (-12.45,-3.45) {$\nu_2$}; 
\node [above] at (-11.45,-3.45) {$\nu_1$}; 
\node [left] at (-15.5,-3.5) {(c) $\frak{b}_l :$}; 
\draw (-14.9,-3.5) -- (-14.5,-3.5); 
\draw [dotted] (-14.5,-3.5) -- (-14,-3.5); 
\draw (-14,-3.5) -- (-13.6,-3.5); 
\draw (-13.4,-3.5) -- (-12.6,-3.5);
\draw (-11.6,-3.5) -- (-11.7,-3.4); 
\draw (-11.6,-3.5) -- (-11.7,-3.6); 
\draw (-12.4,-3.475) -- (-11.65,-3.475); 
\draw (-12.4,-3.525) -- (-11.65,-3.525);

\draw (-8,-3.5) circle [radius = 0.1]; 
\draw (-6.5,-3.5) circle [radius = 0.1]; 
\filldraw [black] (-5.5,-3.5) circle [radius = 0.1]; 
\filldraw [black] (-4.5,-3.5) circle [radius = 0.1]; 
\draw (-3.5,-3.5) circle [radius = 0.1]; 
\draw (-2,-3.5) circle [radius = 0.1]; 
\draw (-1,-3.5) circle [radius = 0.1]; 
\node [above] at (-7.95,-3.45) {$\phi_2$}; 
\node [above] at (-6.25,-3.45) {$\phi_{i-2}$}; 
\node [above] at (-5.45,-3.45) {$\nu_2$}; 
\node [above] at (-4.45,-3.45) {$\nu_1$}; 
\node [above] at (-3.45,-3.45) {$\phi_i$}; 
\node [above] at (-1.75,-3.45) {$\phi_{l-1}$}; 
\node [above] at (-0.95,-3.45) {$\phi_l$}; 
\node [left] at (-8,-3.5) {(d) $\frak{b}_l :$}; 
\draw (-7.9,-3.5) -- (-7.5,-3.5); 
\draw [dotted] (-7.5,-3.5) -- (-7,-3.5); 
\draw (-7,-3.5) -- (-6.6,-3.5); 
\draw (-6.4,-3.5) -- (-5.6,-3.5);
\draw (-5.4,-3.5) -- (-4.6,-3.5);
\draw (-4.4,-3.5) -- (-3.6,-3.5);
\draw (-3.4,-3.5) -- (-3,-3.5);
\draw [dotted] (-3,-3.5) -- (-2.5,-3.5);
\draw (-2.5,-3.5) -- (-2.1,-3.5);
\draw (-1.1,-3.5) -- (-1.2,-3.4); 
\draw (-1.1,-3.5) -- (-1.2,-3.6); 
\draw (-1.9,-3.475) -- (-1.15,-3.475); 
\draw (-1.9,-3.525) -- (-1.15,-3.525);

\filldraw [black] (1.5,-3.5) circle [radius = 0.1]; 
\draw (2.5,-3.5) circle [radius = 0.1]; 
\draw (4,-3.5) circle [radius = 0.1]; 
\draw (5,-3.5) circle [radius = 0.1]; 
\node [above] at (1.55,-3.45) {$\nu$}; 
\node [above] at (2.55,-3.45) {$\phi_2$}; 
\node [above] at (4.25,-3.45) {$\phi_{l-1}$}; 
\node [above] at (5.05,-3.45) {$\phi_l$}; 
\node [left] at (1.5,-3.5) {(e) $\frak{b}_l :$}; 
\draw (1.6,-3.5) -- (2.4,-3.5); 
\draw (2.6,-3.5) -- (3,-3.5); 
\draw [dotted] (3,-3.5) -- (3.5,-3.5); 
\draw (3.5,-3.5) -- (3.9,-3.5); 
\draw (4.9,-3.5) -- (4.8,-3.4); 
\draw (4.9,-3.5) -- (4.8,-3.6); 
\draw (4.1,-3.475) -- (4.85,-3.475); 
\draw (4.1,-3.525) -- (4.85,-3.525);

\end{tikzpicture}
\end{center}
\end{figure}
\end{landscape}

\begin{landscape} 
\begin{table} 
\caption{Poincar\'{e} Polynomials of The Irreducible Unitary Representations $A_\frak{q}$ of $SO_0(2,2l-2)(l \ge 3):$}\label{d-table}
\begin{tabular}{||c|c|c|c|c|c||}
\hline
\begin{tabular}{c} $\Delta(\frak{u} \cap \frak{p}_-),$\\$R_-(\frak{q})$ \end{tabular} & $\Phi_\frak{q}$ & 
\begin{tabular}{c} $\Delta(\frak{u} \cap \frak{p}_+),$\\$R_+(\frak{q})$ \end{tabular} & $\Gamma$ & $Y_\frak{q}$ & $P_\frak{q}(x, t)$ \\
\hline
\hline
\begin{tabular}{c} empty, \\$0$ \end{tabular} & 
\begin{tabular}{c} $\{\nu, \phi_2, \ldots ,\phi_l\};$\\$\nu=\phi_1.$\\ See (a).  \end{tabular} & 
\begin{tabular}{c} $\{\beta \in \Delta_n^+ : \beta \ge $\\$\nu\},2l-2$\\ \\ 
$\{\beta \in \Delta_n^+ : \beta \ge $\\$\phi_1+\cdots +\phi_i\},2l-$\\$i-1(2 \le i \le l-2)$\\ \\$\{\beta \in \Delta_n^+ : \beta \ge $\\$\phi_1+\cdots +\phi_{l-2}+\phi_a\},$\\$l-1(a=l-1,l)$\\ \\
$\{\beta \in \Delta_n^+ : \beta \ge $\\$\xi_1 \textrm{ or } \xi_2\},l$\\ \\empty, $0$ \end{tabular} & 
\begin{tabular}{c} $\{\nu\}$\\ \\$\{\phi_i\}$\\$(2 \le $\\$i \le l-2)$\\ \\$\{\phi_a\}$\\$(a=l-1,l)$\\ \\$\{\phi_{l-1},\phi_l\}$\\ \\ empty \end{tabular} & 
\begin{tabular}{c} singleton \\ \\$\frac{SU(i)}{S(U(1)\times U(i-1))}$\\$(2 \le i \le l-2)$\\ \\$\frac{SU(l)}{S(U(1)\times U(l-1))}$\\ \\$\frac{SU(l-1)}{S(U(1)\times U(l-2))}$\\  \\
$\frac{SO(2l)}{SO(2)\times SO(2l-2)}$ \end{tabular} & 
\begin{tabular}{c} $x^{2l-2}$\\ \\$x^{2l-i-1}$\\$+x^{2l-i}t+$\\$\cdots+x^{2l-2}t^{i-1}$\\ \\$x^{l-1}+x^lt+$\\$\cdots+x^{2l-2}t^{l-1}$\\ \\ 
$x^l+x^{l+1}t+$\\$\cdots+x^{2l-2}t^{l-2}$\\ \\$1+xt+\cdots $\\$+2x^{l-1}t^{l-1}+$\\$\cdots+x^{2l-2}t^{2l-2}$  \end{tabular} \\  
\hline 
\begin{tabular}{c} $-\{\beta \in \Delta_n^+ : \beta \le $\\$\phi_1+\cdots +\phi_i\},$\\$i$\\$(1 \le i \le l-3)$\\$(l \ge 4)$ \end{tabular} & 
\begin{tabular}{c} $\{\nu_1,\nu_2,\phi_2,\ldots ,$\\$\phi_i,\phi_{i+2},\ldots , \phi_l\};\nu_1$\\$=-(\phi_1+\cdots +\phi_i),$\\$\nu_2=\phi_1+\cdots+\phi_{i+1}.$\\ See (b). \end{tabular} & 
\begin{tabular}{c} $\{\beta \in \Delta_n^+ : \beta \ge $\\$\nu_2\},2l-2-i$\\ \\$\{\beta \in \Delta_n^+ : \beta \ge $\\$\phi_1+\cdots +\phi_j\},2l-j$\\$-1(i+2 \le j \le l-2)$\\ \\
$\{\beta \in \Delta_n^+ : \beta \ge $\\$\phi_1+\cdots +\phi_{l-2}+\phi_a\},$\\$l-1(a=l-1,l)$\\ \\$\{\beta \in \Delta_n^+ : \beta \ge $\\$\xi_1 \textrm{ or } \xi_2\},l$ \end{tabular} & 
\begin{tabular}{c} $\{\nu_1,\nu_2\}$\\ \\$\{\nu_1,\phi_j\}$\\$(i+2\le $\\$ j \le l-2)$\\ \\$\{\nu_1,\phi_a\}$\\$(a=l-1,l)$\\ \\$\{\nu_1,\phi_{l-1},\phi_l\}$ \end{tabular} & 
\begin{tabular}{c} singleton \\ \\$\frac{SU(j-i)}{S(U(1)\times U(j-i-1))}$\\$(i+2 \le $\\$ j \le l-2)$\\ \\$\frac{SU(l-i)}{S(U(1)\times U(l-i-1))}$\\ \\$\frac{SU(l-1-i)}{S(U(1)\times U(l-2-i))}$ \end{tabular} & 
\begin{tabular}{c} $x^{2l-2-i}t^i$\\ \\$x^{2l-j-1}t^i$\\$+x^{2l-j}t^{i+1}+$\\$\cdots+x^{2l-2-i}t^{j-1}$\\ \\$x^{l-1}t^i+x^lt^{i+1}+$\\$\cdots+x^{2l-2-i}t^{l-1}$\\ \\ 
$x^lt^i+x^{l+1}t^{i+1}+$\\$\cdots+x^{2l-2-i}t^{l-2}$  \end{tabular} \\  
\hline 
\end{tabular} 
\end{table} 
\end{landscape}

\begin{landscape} 
\begin{table} 
\begin{tabular}{||c|c|c|c|c|c||}
\hline
\begin{tabular}{c} $\Delta(\frak{u} \cap \frak{p}_-),$\\$R_-(\frak{q})$ \end{tabular} & $\Phi_\frak{q}$ & 
\begin{tabular}{c} $\Delta(\frak{u} \cap \frak{p}_+),$\\$R_+(\frak{q})$ \end{tabular} & $\Gamma$ & $Y_\frak{q}$ & $P_\frak{q}(x,t)$ \\
\hline
\hline
,, & 
,, & 
\begin{tabular}{c} $\{\beta \in \Delta_n^+ : \beta \ge $\\$\phi_1+\cdots +\phi_i+$\\$2\phi_{i+1}+\cdots +2\phi_{l-2}$\\$+\phi_{l-1}+\phi_l\},i$ \end{tabular} & 
$\{\nu_1\}$  & 
$\frac{SO(2l-2i)}{SO(2)\times SO(2l-2i-2)}$ & 
\begin{tabular}{c} $x^it^i+x^{i+1}t^{i+1}$\\$+\cdots +2x^{l-1}t^{l-1}$\\$+\cdots+$\\$x^{2l-i-2}t^{2l-i-2}$  \end{tabular} \\  
\hline 
\begin{tabular}{c} $-\{\beta \in \Delta_n^+ : \beta \le $\\$\phi_1+\cdots +\phi_{l-2}\},$\\$l-2$ \end{tabular} & 
\begin{tabular}{c} $\{\nu_1,\nu_2,\nu'_2,$\\$\phi_2,\ldots ,$\\$\phi_{l-2}\};\nu_1=-$\\$(\phi_1+\cdots +\phi_{l-2}),$\\$\nu_2=\xi_1,\nu'_2=\xi_2.$\\ See (c). \end{tabular} & 
\begin{tabular}{c} $\{\beta \in \Delta_n^+ : \beta \ge $\\$\nu_2, \textrm{or }\nu'_2\},l$\\ \\$\{\beta \in \Delta_n^+ : \beta \ge \xi_i\},$\\$l-1(i =1,2)$\\ \\
$\{\beta \in \Delta_n^+ : \beta \ge \phi_1+$\\$\cdots+\phi_l\},l-2$ \end{tabular} & 
\begin{tabular}{c} $\{\nu_1,\nu_2,\nu'_2\}$\\ \\$\{\nu_1,\xi_i\}$\\$(i=1,2)$\\ \\$\{\nu_1\}$ \end{tabular} & 
\begin{tabular}{c} singleton \\ \\ $\frac{SU(2)}{S(U(1)\times U(1))}$\\ \\$\frac{SO(4)}{SO(2)\times SO(2)}$ \end{tabular} & 
\begin{tabular}{c} $x^lt^{l-2}$\\ \\$x^{l-1}t^{l-2}$\\$+x^lt^{l-1}$\\ \\$x^{l-2}t^{l-2}+$\\$2x^{l-1}t^{l-1}$\\$+x^lt^l$   \end{tabular} \\  
\hline
\begin{tabular}{c} $-\{\beta \in \Delta_n^+ : \beta \le $\\$\phi_1+\cdots +\phi_{l-2}$\\$+\phi_a\},l-1$\\$(a=l-1,l)$ \end{tabular} & 
\begin{tabular}{c} $\{\nu_1,\nu_2,\phi_2,$\\$\ldots ,\phi_{l-2},\phi_a\};\nu_1$\\$=-(\phi_1+\cdots +$\\$\phi_{l-2}+\phi_a),$\\$\nu_2=\phi_1$\\$+\cdots +\phi_{l-2}+\phi_b$\\
$(b\in\{l-1,l\}$\\$\setminus \{a\}).$\\ See (d). \end{tabular} & 
\begin{tabular}{c} $\{\beta \in \Delta_n^+ : \beta \ge \nu_2\},$\\$l-1$\\ \\$\{\beta \in \Delta_n^+ : \beta \ge$\\$\phi_1+\cdots+\phi_l\},l-2$\\ \\
$\{\beta \in \Delta_n^+ : \beta \ge \phi_1+$\\$\cdots +\phi_{j-1}+2\phi_j+\cdots $\\$+2\phi_{l-2}+\phi_{l-1}+\phi_l\},$\\$j-1$\\$(2\le j\le l-2, l\ge 4)$\\ \\empty, $0$ \end{tabular} & 
\begin{tabular}{c} $\{\nu_1,\nu_2\}$\\ \\$\{\nu_1,\phi_a\}$\\ \\$\{\nu_1,\phi_j\}$\\$(2\le j\le l-2,$\\$ l\ge 4)$\\ \\$\{\nu_1\}$ \end{tabular} & 
\begin{tabular}{c} singleton \\ \\ $\frac{SU(2)}{S(U(1)\times U(1))}$\\ \\$\frac{SU(l-j+1)}{S(U(1)\times U(l-j))}$\\$(2\le j\le l-2,$\\$ l\ge 4)$\\ \\$\frac{SU(l)}{S(U(1)\times U(l-1))}$  \end{tabular} & 
\begin{tabular}{c} $x^{l-1}t^{l-1}$\\ \\$x^{l-2}t^{l-1}$\\$+x^{l-1}t^l$\\ \\$x^{j-1}t^{l-1}+$\\$x^jt^l+\cdots+$\\$x^{l-1}t^{2l-j-1}$\\ \\$t^{l-1}+$\\$xt^l+\cdots$\\$+x^{l-1}t^{2l-2}$   \end{tabular} \\  
\hline
\end{tabular} 
\end{table} 
\end{landscape}

\begin{landscape} 
\begin{table} 
\begin{tabular}{||c|c|c|c|c|c||}
\hline
\begin{tabular}{c} $\Delta(\frak{u} \cap \frak{p}_-),$\\$R_-(\frak{q})$ \end{tabular} & $\Phi_\frak{q}$ & 
\begin{tabular}{c} $\Delta(\frak{u} \cap \frak{p}_+),$\\$R_+(\frak{q})$ \end{tabular} & $\Gamma$ & $Y_\frak{q}$ & $P_\frak{q}(x,t)$ \\
\hline
\hline
\begin{tabular}{c} $-\{\beta \in \Delta_n^+ : \beta \le $\\$\xi_1,\textrm{or }\xi_2\},$\\$l$ \end{tabular} & 
\begin{tabular}{c} $\{\nu_1,\nu'_1,\nu_2,\phi_2,$\\$\ldots ,\phi_{l-2},\};\nu_1$\\$=-\xi_1,\nu'_1=$\\$-\xi_2,\nu_2=$\\$\phi_1+\cdots +\phi_l.$\\ See (e). \end{tabular} & 
\begin{tabular}{c} $\{\beta \in \Delta_n^+ : \beta \ge \nu_2\},$\\$l-2$\\ \\
$\{\beta \in \Delta_n^+ : \beta \ge \phi_1+$\\$\cdots +\phi_{j-1}+2\phi_j+\cdots $\\$+2\phi_{l-2}+\phi_{l-1}+\phi_l\},$\\$j-1$\\$(2\le j\le l-2, l\ge 4)$\\ \\empty, $0$ \end{tabular} & 
\begin{tabular}{c} $\{\nu_1,\nu'_1,\nu_2\}$\\ \\$\{\nu_1,\nu'_1,\phi_j\}$\\$(2\le j\le $\\$l-2,$\\$ l\ge 4)$\\ \\$\{\nu_1,\nu'_1\}$ \end{tabular} & 
\begin{tabular}{c} singleton \\ \\$\frac{SU(l-j)}{S(U(1)\times U(l-j-1))}$\\$(2\le j\le l-2,$\\$ l\ge 4)$\\ \\$\frac{SU(l-1)}{S(U(1)\times U(l-2))}$  \end{tabular} & 
\begin{tabular}{c} $x^{l-2}t^l$\\ \\$x^{j-1}t^l+$\\$x^jt^{l+1}+\cdots+$\\$x^{l-2}t^{2l-j-1}$\\ \\$t^l+$\\$xt^{l+1}+\cdots$\\$+x^{l-2}t^{2l-2}$   \end{tabular} \\  
\hline 
 \begin{tabular}{c} $-\{\beta \in \Delta_n^+ : \beta \le $\\$\phi_1+\cdots+\phi_l\},$\\$l+1$\\$(l \ge 4)$ \end{tabular} & 
\begin{tabular}{c} $\{\nu_1,\nu_2,\phi_2,$\\$\ldots ,\phi_{l-3},\phi_{l-1}, \phi_l\};$\\$\nu_1=-(\phi_1+\cdots+\phi_l),$\\$\nu_2=\phi_1+\cdots +\phi_{l-3}$\\$+2\phi_{l-2}+\phi_{l-1}+\phi_l.$\\ See (f). \end{tabular} & 
\begin{tabular}{c} $\{\beta \in \Delta_n^+ : \beta \ge \nu_2\},$\\$l-3$\\ \\
$\{\beta \in \Delta_n^+ : \beta \ge \phi_1+$\\$\cdots +\phi_{j-1}+2\phi_j+\cdots $\\$+2\phi_{l-2}+\phi_{l-1}+\phi_l\},$\\$j-1$\\$(2\le j\le l-3)$\\ \\empty, $0$ \end{tabular} & 
\begin{tabular}{c} $\{\nu_1,\nu_2\}$\\ \\$\{\nu_1,\phi_j\}$\\$(2\le j\le $\\$l-3)$\\ \\$\{\nu_1\}$ \end{tabular} & 
\begin{tabular}{c} singleton \\ \\$\frac{SU(l-j-1)}{S(U(1)\times U(l-j-2))}$\\$(2\le j\le l-3)$\\ \\$\frac{SU(l-2)}{S(U(1)\times U(l-3))}$  \end{tabular} & 
\begin{tabular}{c} $x^{l-3}t^{l+1}$\\ \\$x^{j-1}t^{l+1}+$\\$x^jt^{l+2}+\cdots+$\\$x^{l-3}t^{2l-j-1}$\\ \\$t^{l+1}+$\\$xt^{l+2}+\cdots$\\$+x^{l-3}t^{2l-2}$   \end{tabular} \\  
\hline 
\begin{tabular}{c} $-\{\beta \in \Delta_n^+ : \beta \le $\\$\phi_1+\cdots +\phi_{i-1}$\\$+2\phi_i+\cdots+2\phi_{l-2}$\\$+\phi_{l-1}+\phi_l\},$\\$2l-i$\\$(3\le i \le l-2)$\\$(l \ge 4)$ \end{tabular} & 
\begin{tabular}{c} $\{\nu_1,\nu_2,\phi_2, \ldots,$\\$\phi_{i-2},\phi_i,\ldots ,\phi_l\};$\\
$\nu_1=-(\phi_1+\cdots +\phi_{i-1}$\\$+2\phi_i+\cdots+2\phi_{l-2}$\\$+\phi_{l-1}+\phi_l),\nu_2$\\$= \phi_1+\cdots+\phi_{i-2}$\\$+2\phi_{i-1}+\cdots+$\\$2\phi_{l-2}+\phi_{l-1}+\phi_l.$\\ See (g). \end{tabular} & 
\begin{tabular}{c} $\{\beta \in \Delta_n^+ : \beta \ge $\\$\nu_2\},i-2$\\ \\
$\{\beta \in \Delta_n^+ : \beta \ge $\\$\phi_1+\cdots +\phi_{j-1}$\\$+2\phi_j+\cdots+2\phi_l\},$\\$j-1$\\$(2 \le j \le i-2)$\\ \\ empty, $0$ \end{tabular} & 
\begin{tabular}{c} $\{\nu_1,\nu_2\}$\\ \\$\{\nu_1,\phi_j\}$\\$(2\le j \le $\\$ i-2)$\\ \\$\{\nu_1\}$ \end{tabular} & 
\begin{tabular}{c} singleton \\ \\$\frac{SU(i-j)}{S(U(1)\times U(i-j-1))}$\\$(2 \le j \le i-2)$\\ \\$\frac{SU(i-1)}{S(U(1)\times U(i-2))}$  \end{tabular} & 
\begin{tabular}{c} $x^{i-2}t^{2l-i}$\\ \\$x^{j-1}t^{2l-i}+$\\$x^jt^{2l-i+1}+\cdots$\\$+x^{i-2}t^{2l-j-1}$\\ \\$t^{2l-i}+$\\$xt^{2l-i+1}+\cdots$\\$+x^{i-2}t^{2l-2}$   \end{tabular} \\   
\hline
\begin{tabular}{c} $-\{\beta \in \Delta_n^+ : \beta \le $\\$\phi_1+2\phi_2+\cdots+$\\$2\phi_{l-2}+\phi_{l-1}+\phi_l\},$\\$2l-2$ \end{tabular} & 
\begin{tabular}{c} $\{\nu, \phi_2,\ldots ,\phi_l\}; \nu=$\\$-(\phi_1+2\phi_2+\cdots+$\\$2\phi_{l-2}+\phi_{l-1}+\phi_l).$\\ See (h). \end{tabular} & 
\begin{tabular}{c} empty, $0$ \end{tabular} & 
$\{\nu \}$  & singleton & $t^{2l-2}$  \\ 
\hline 
\end{tabular} 
\end{table}
\end{landscape}

\begin{landscape}
\begin{center} 
\begin{tikzpicture} 

\filldraw [black] (0,0) circle [radius = 0.1]; 
\draw (1,0) circle [radius = 0.1]; 
\draw (2.5,0) circle [radius = 0.1]; 
\draw (4,0) circle [radius = 0.1]; 
\draw (4.9,0.5) circle [radius = 0.1]; 
\draw (4.9,-0.5) circle [radius = 0.1]; 
\node [above] at (0.05,0.05) {$\nu$}; 
\node [above] at (1.05,0.05) {$\phi_2$}; 
\node [above] at (2.55,0.05) {$\phi_i$};  
\node [above] at (3.85,0.05) {$\phi_{l-2}$}; 
\node [above] at (4.95,0.55) {$\phi_{l-1}$}; 
\node [below] at (4.95,-0.55) {$\phi_l$}; 
\node [left] at (-0.4,0) {(a)  $\frak{\delta}_l :$}; 
\node [left] at (-0.5,-0.5) {$(l \ge 3)$}; 
\draw (0.1,0) -- (0.9,0); 
\draw (1.1,0) -- (1.5,0); 
\draw [dotted] (1.5,0) -- (2,0); 
\draw (2,0) -- (2.4,0); 
\draw (2.6,0) -- (3,0); 
\draw [dotted] (3,0) -- (3.5,0); 
\draw (3.5,0) -- (3.9,0); 
\draw (4.1,0) -- (4.85,0.45); 
\draw (4.1,0) -- (4.85,-0.45);

\draw (8,0) circle [radius = 0.1]; 
\draw (9.5,0) circle [radius = 0.1]; 
\filldraw [black] (10.5,0) circle [radius = 0.1]; 
\filldraw [black] (11.5,0) circle [radius = 0.1]; 
\draw (12.5,0) circle [radius = 0.1]; 
\draw (14,0) circle [radius = 0.1]; 
\draw (14.9,0.5) circle [radius = 0.1]; 
\draw (14.9,-0.5) circle [radius = 0.1]; 
\node [above] at (8.05,0.05) {$\phi_2$}; 
\node [above] at (9.55,0.05) {$\phi_i$}; 
\node [above] at (10.55,0.05) {$\nu_1$};  
\node [above] at (11.55,0.05) {$\nu_2$}; 
\node [above] at (12.65,0.05) {$\phi_{i+2}$}; 
\node [above] at (13.85,0.05) {$\phi_{l-2}$}; 
\node [above] at (14.95,0.55) {$\phi_{l-1}$}; 
\node [below] at (14.95,-0.55) {$\phi_l$}; 
\node [left] at (7.6,0) {(b)  $\frak{\delta}_l :$}; 
\node [left] at (7.5,-0.5) {$(l \ge 4)$}; 
\draw (8.1,0) -- (8.5,0); 
\draw [dotted] (8.5,0) -- (9,0); 
\draw (9,0) -- (9.4,0); 
\draw (9.6,0) -- (10.4,0); 
\draw (10.6,0) -- (11.4,0); 
\draw (11.6,0) -- (12.4,0); 
\draw (12.6,0) -- (13,0); 
\draw [dotted] (13,0) -- (13.5,0); 
\draw (13.5,0) -- (13.9,0); 
\draw (14.1,0) -- (14.85,0.45); 
\draw (14.1,0) -- (14.85,-0.45);

\draw (17,0) circle [radius = 0.1]; 
\draw (18.5,0) circle [radius = 0.1]; 
\filldraw [black] (19.5,0) circle [radius = 0.1]; 
\filldraw [black] (20.4,0.5) circle [radius = 0.1]; 
\filldraw [black] (20.4,-0.5) circle [radius = 0.1]; 
\node [above] at (17.05,0.05) {$\phi_2$}; 
\node [above] at (18.65,0.05) {$\phi_{l-2}$};  
\node [above] at (19.45,0.05) {$\nu_1$}; 
\node [above] at (20.45,0.55) {$\nu_2$}; 
\node [below] at (20.45,-0.55) {$\nu'_2$}; 
\node [left] at (16.7,0) {(c)  $\frak{\delta}_l :$}; 
\node [left] at (16.8,-0.5) {$(l \ge 3)$}; 
\draw (17.1,0) -- (17.5,0); 
\draw [dotted] (17.5,0) -- (18,0); 
\draw (18,0) -- (18.4,0); 
\draw (18.6,0) -- (19.4,0); 
\draw (19.6,0) -- (20.35,0.45); 
\draw (19.6,0) -- (20.35,-0.45);

\draw (0,-2) circle [radius = 0.1]; 
\draw (1.5,-2) circle [radius = 0.1]; 
\draw (2.5,-2) circle [radius = 0.1]; 
\filldraw [black] (3.4,-1.5) circle [radius = 0.1]; 
\filldraw [black] (3.4,-2.5) circle [radius = 0.1]; 
\node [above] at (0.05,-1.95) {$\phi_2$}; 
\node [above] at (1.65,-1.95) {$\phi_{l-2}$};  
\node [above] at (2.45,-1.95) {$\phi_a$}; 
\node [above] at (3.45,-1.45) {$\nu_2$}; 
\node [below] at (3.45,-2.55) {$\nu_1$}; 
\node [left] at (-0.4,-2) {(d)  $\frak{\delta}_l :$}; 
\node [left] at (-0.5,-2.5) {$(l \ge 3)$}; 
\draw (0.1,-2) -- (0.5,-2); 
\draw [dotted] (0.5,-2) -- (1,-2); 
\draw (1,-2) -- (1.4,-2); 
\draw (1.6,-2) -- (2.4,-2); 
\draw (2.6,-2) -- (3.35,-1.55); 
\draw (2.6,-2) -- (3.35,-2.45);

\draw (6.5,-2) circle [radius = 0.1]; 
\draw (8,-2) circle [radius = 0.1]; 
\filldraw [black] (9,-2) circle [radius = 0.1]; 
\filldraw [black] (9.9,-1.5) circle [radius = 0.1]; 
\filldraw [black] (9.9,-2.5) circle [radius = 0.1]; 
\node [above] at (6.55,-1.95) {$\phi_2$}; 
\node [above] at (8.15,-1.95) {$\phi_{l-2}$};  
\node [above] at (9.05,-1.95) {$\nu_2$}; 
\node [above] at (9.95,-1.45) {$\nu_1$}; 
\node [below] at (9.95,-2.55) {$\nu'_1$}; 
\node [left] at (6.1,-2) {(e)  $\frak{\delta}_l :$}; 
\node [left] at (6,-2.5) {$(l \ge 3)$}; 
\draw (6.6,-2) -- (7,-2); 
\draw [dotted] (7,-2) -- (7.5,-2); 
\draw (7.5,-2) -- (7.9,-2); 
\draw (8.1,-2) -- (8.9,-2); 
\draw (9.1,-2) -- (9.85,-1.55); 
\draw (9.1,-2) -- (9.85,-2.45);

\draw (13,-2) circle [radius = 0.1]; 
\draw (14.5,-2) circle [radius = 0.1]; 
\filldraw [black] (15.5,-2) circle [radius = 0.1]; 
\filldraw [black] (16.5,-2) circle [radius = 0.1]; 
\draw (17.4,-1.5) circle [radius = 0.1]; 
\draw (17.4,-2.5) circle [radius = 0.1]; 
\node [above] at (13.05,-1.95) {$\phi_2$}; 
\node [above] at (14.65,-1.95) {$\phi_{l-3}$}; 
\node [above] at (15.55,-1.95) {$\nu_2$};  
\node [above] at (16.55,-1.95) {$\nu_1$}; 
\node [above] at (17.35,-1.45) {$\phi_{l-1}$}; 
\node [below] at (17.35,-2.55) {$\phi_l$}; 
\node [left] at (12.6,-2) {(f)  $\frak{\delta}_l :$}; 
\node [left] at (12.5,-2.5) {$(l \ge 4)$}; 
\draw (13.1,-2) -- (13.5,-2); 
\draw [dotted] (13.5,-2) -- (14,-2); 
\draw (14,-2) -- (14.4,-2); 
\draw (14.6,-2) -- (15.4,-2); 
\draw (15.6,-2) -- (16.4,-2); 
\draw (16.6,-2) -- (17.35,-1.55); 
\draw (16.6,-2) -- (17.35,-2.45);

\draw (0,-4) circle [radius = 0.1]; 
\draw (1.5,-4) circle [radius = 0.1]; 
\filldraw [black] (2.5,-4) circle [radius = 0.1]; 
\filldraw [black] (3.5,-4) circle [radius = 0.1]; 
\draw (4.5,-4) circle [radius = 0.1]; 
\draw (6,-4) circle [radius = 0.1]; 
\draw (6.9,-3.5) circle [radius = 0.1]; 
\draw (6.9,-4.5) circle [radius = 0.1]; 
\node [above] at (0.05,-3.95) {$\phi_2$}; 
\node [above] at (1.65,-3.95) {$\phi_{i-2}$}; 
\node [above] at (2.55,-3.95) {$\nu_2$}; 
\node [above] at (3.55,-3.95) {$\nu_1$}; 
\node [above] at (4.55,-3.95) {$\phi_i$};  
\node [above] at (5.85,-3.95) {$\phi_{l-2}$}; 
\node [above] at (6.95,-3.45) {$\phi_{l-1}$}; 
\node [below] at (6.95,-4.55) {$\phi_l$}; 
\node [left] at (-0.4,-4) {(g)  $\frak{\delta}_l :$}; 
\node [left] at (-0.5,-4.5) {$(l \ge 4)$}; 
\draw (0.1,-4) -- (0.5,-4); 
\draw [dotted] (0.5,-4) -- (1,-4); 
\draw (1,-4) -- (1.4,-4); 
\draw (1.6,-4) -- (2.4,-4); 
\draw (2.6,-4) -- (3.4,-4); 
\draw (3.6,-4) -- (4.4,-4); 
\draw (4.6,-4) -- (5,-4); 
\draw [dotted] (5,-4) -- (5.5,-4); 
\draw (5.5,-4) -- (5.9,-4); 
\draw (6.1,-4) -- (6.85,-3.55); 
\draw (6.1,-4) -- (6.85,-4.45);

\filldraw [black] (10,-4) circle [radius = 0.1]; 
\draw (11,-4) circle [radius = 0.1]; 
\draw (12.5,-4) circle [radius = 0.1]; 
\draw (13.4,-3.5) circle [radius = 0.1]; 
\draw (13.4,-4.5) circle [radius = 0.1]; 
\node [above] at (10.05,-3.95) {$\nu$}; 
\node [above] at (11.05,-3.95) {$\phi_2$}; 
\node [above] at (12.35,-3.95) {$\phi_{l-2}$}; 
\node [above] at (13.45,-3.45) {$\phi_{l-1}$}; 
\node [below] at (13.45,-4.55) {$\phi_l$}; 
\node [left] at (9.6,-4) {(h)  $\frak{\delta}_l :$}; 
\node [left] at (9.5,-4.5) {$(l \ge 3)$}; 
\draw (10.1,-4) -- (10.9,-4); 
\draw (11.1,-4) -- (11.5,-4); 
\draw [dotted] (11.5,-4) -- (12,-4); 
\draw (12,-4) -- (12.4,-4); 
\draw (12.6,-4) -- (13.35,-3.55); 
\draw (12.6,-4) -- (13.35,-4.45);

\end{tikzpicture} 
\end{center}

\subsection{Poincar\'{e} polynomials of cohomologies of the irreducible unitary representations $A_\frak{q}$ of $SO_0(2,2)$} Let $\frak{g}=\frak{\delta}_2.$ \\ 
Note that $\Delta_n^+=\{\phi_1,\phi_2\}.$
\begin{center}
\begin{tikzpicture}
\filldraw [black] (0,0) circle [radius = 0.1]; 
\filldraw [black] (1,0) circle [radius = 0.1]; 
\node [above] at (0.05,0) {$\phi_1$}; 
\node [above] at (1.05,0) {$\phi_2$}; 
\node [left] at (-0.4,0) {$\frak{\delta}_2 :$}; 
\end{tikzpicture}
\end{center}
Now $\Delta(\frak{u}\cap\frak{p}_-)=\phi, \Delta(\frak{u}\cap \frak{p}_+)=\phi \implies Y_\frak{q}=\frac{SO(4)}{SO(2)\times SO(2)}, P_\frak{q}(x,t)=1+2xt+x^2t^2;$\\ 
$\Delta(\frak{u}\cap\frak{p}_-)=\phi, \Delta(\frak{u}\cap \frak{p}_+)=\{\phi_i\}(i=1,2) \implies Y_\frak{q}=\frac{SU(2)}{S(U(1)\times U(1))}, P_\frak{q}(x,t)=x+x^2t;$\\
$\Delta(\frak{u}\cap\frak{p}_-)=\phi, \Delta(\frak{u}\cap \frak{p}_+)=\{\phi_1,\phi_2\} \implies Y_\frak{q}=$ singleton, $P_\frak{q}(x,t)=x^2;$\\
$\Delta(\frak{u}\cap\frak{p}_-)=\{-\phi_i\}(i=1,2), \Delta(\frak{u}\cap \frak{p}_+)=\phi \implies Y_\frak{q}=\frac{SU(2)}{S(U(1)\times U(1))}, P_\frak{q}(x,t)=t+xt^2;$\\
$\Delta(\frak{u}\cap\frak{p}_-)=\{-\phi_i\}(i=1,2), \Delta(\frak{u}\cap \frak{p}_+)=\{\phi_j\}(j=1,2,j\neq i) \implies Y_\frak{q}=$ singleton, $P_\frak{q}(x,t)=xt;$  \\
$\Delta(\frak{u}\cap\frak{p}_-)=\{-\phi_1,-\phi_2\}, \Delta(\frak{u}\cap \frak{p}_+)=\phi \implies Y_\frak{q}=$ singleton, $P_\frak{q}(x,t)=t^2.$  \\

\end{landscape}

\section*{acknowledgement}

Both authors acknowledge the financial support from the Department of Science and Technology (DST), Govt. of India under the Scheme 
"Fund for Improvement of S\&T Infrastructure (FIST)" [File No. SR/FST/MS-I/2019/41]. 
Ankita Pal acknowledges the financial support from Council of Scientific and Industrial Research (CSIR) [File No. 08/155(0091)/2021-EMR-I].

\end{document}